\documentclass[10pt]{amsart}

\usepackage{amsmath, amsthm}
\usepackage{eucal}

\usepackage{amssymb}
\usepackage{amscd}
\usepackage{palatino}
\usepackage{latexsym}
\usepackage{epsfig}
\usepackage{graphicx}
\usepackage{amsfonts}
\usepackage{psfrag}

\usepackage{url}
\usepackage{cite}

\usepackage[margin=1in]{geometry}

\input xy
\xyoption{all}
\UseComputerModernTips

\numberwithin{equation}{section}
\numberwithin{figure}{section}

\newtheorem{theorem}{Theorem}[section]
\newtheorem{lemma}[theorem]{Lemma}
\newtheorem{proposition}[theorem]{Proposition}

\newtheorem{fact}[theorem]{Fact}

\newtheorem{remark}[theorem]{Remark}
\newtheorem{example}[theorem]{Example}
\newtheorem{question}[theorem]{Question}

\theoremstyle{definition}
\newtheorem{definition}[theorem]{Definition}

\newcommand\toe[2]{\mathcal{T}_{#1}(#2)}
\newcommand\head[2]{\mathcal{H}_{#1}(#2)}

\newcommand{\C}{{\mathbb{C}}}

\newcommand{\Z}{{\mathbb{Z}}}

\newcommand{\N}{{\mathbb{N}}}

\renewcommand{\t}{\mathfrak{t}}

\newcommand{\into}{\hookrightarrow}

\DeclareMathOperator{\Lie}{Lie}

\DeclareMathOperator{\pt}{pt}

\newcommand{\hsm}{{\hspace{1mm}}}

\usepackage{multicol}
\usepackage[dvips]{color}

\definecolor{gold}{rgb}{0.85,.66,0}
\definecolor{cherry}{rgb}{0.9,.1,.2}
\definecolor{burgundy}{rgb}{0.8,.2,.2}
\definecolor{orangered}{rgb}{0.85,.3,0}
\definecolor{orange}{rgb}{0.85,.4,0}
\definecolor{olive}{rgb}{.45,.4,0}
\definecolor{lime}{rgb}{.6,.9,0}
\definecolor{green}{rgb}{.2,.7,0}
\definecolor{grey}{rgb}{.4,.4,.2}
\definecolor{brown}{rgb}{.4,.3,.1}

\DeclareMathOperator{\Hess}{Hess}
\newcommand{\Fill}{{\mathcal F}i \ell \ell} 
\newcommand{\PFill}{{\mathcal P}Fi \ell \ell}
\newcommand{\roll}{{ro\ell \ell}}
\newcommand{\Flags}{\mathcal{F}\ell ags}
\DeclareMathOperator{\Sym}{Sym}

\begin{document}

\title{Poset pinball, the dimension pair algorithm, 
and type $A$ regular nilpotent Hessenberg varieties}

\author{Darius Bayegan}
\address{Department of Pure Mathematics and Mathematical 
Statistics\\ Centre for Mathematical Sciences \\ Wilberforce Road \\
Cambridge CB3 0WA \\ United Kingdom} 

\author{Megumi Harada}
\address{Department of Mathematics and
Statistics\\ McMaster University\\ 1280 Main Street West\\ Hamilton, Ontario L8S4K1\\ Canada}
\thanks{The second author is partially supported by an NSERC Discovery Grant,
an NSERC University Faculty Award, and an Ontario Ministry of Research
and Innovation Early Researcher Award.}

\keywords{} 

\date{\today}



\begin{abstract}

  In this manuscript we develop the theory of \textbf{poset pinball},
  a combinatorial game recently introduced by Harada and Tymoczko for
  the study of the equivariant cohomology rings of GKM-compatible subspaces of GKM spaces.  Harada and Tymoczko also prove that in certain circumstances, a \textbf{successful outcome of Betti poset pinball} yields a module basis for the equivariant cohomology ring of the GKM-compatible subspace.
Our main
  contributions are twofold. First we construct an algorithm (which we
  call the \textbf{dimension pair algorithm}) which yields the result of a successful outcome of Betti poset pinball 
  for any type $A$
  regular nilpotent Hessenberg and any type $A$ nilpotent Springer
  variety, considered as GKM-compatible subspaces of the flag variety $\Flags(\C^n)$.  The definition of the algorithm is motivated by a
  correspondence between Hessenberg affine cells and certain Schubert
  polynomials which we learned from Erik Insko. Second, in the special case of the type $A$ regular nilpotent Hessenberg
  varieties specified by the Hessenberg function $h(1)=h(2)=3$ and
  $h(i) = i+1$ for $3 \leq i \leq n-1$ and $h(n)=n$, we prove that 
  the pinball result coming from the dimension pair algorithm is
 \textbf{poset-upper-triangular};
  by results of Harada and Tymoczko this implies the corresponding equivariant cohomology classes form a $H^*_{S^1}(\pt)$-module basis for the
  $S^1$-equivariant cohomology ring of the Hessenberg variety.

\end{abstract}

\maketitle

\setcounter{tocdepth}{1}
\tableofcontents

\section{Introduction}\label{sec:intro}

The purpose of this manuscript is to further develop the theory of 
\textbf{poset pinball}, a combinatorial game introduced in
\cite{HarTym10} for the purpose of computing in equivariant
cohomology rings,\footnote{All cohomology rings in this note are with
  $\C$ coefficients.} in certain cases of \textbf{type $A$ 
  nilpotent Hessenberg varieties}. One of the main uses of poset pinball in \cite{HarTym10} is to construct module bases for the equivariant cohomology rings of \textbf{GKM-compatible subspaces} of GKM spaces \cite[Definition 4.5]{HarTym10}. In the context of this manuscript, the ambient GKM space is the flag variety $\Flags(\C^n)$ equipped with the action of the diagonal subgroup $T$ of $U(n,\C)$, and the GKM-compatible subspaces are the nilpotent Hessenberg varieties. It is well-recorded in the literature (e.g. \cite{Tym05} and references therein) that GKM spaces often have geometrically and/or combinatorially natural module bases for their equivariant cohomology rings; the basis of equivariant Schubert classes $\{\sigma_w\}_{w \in S_n}$ for $H^*_T(\Flags(\C^n))$ is a famous example. The results of this manuscript represent first steps towards the larger goal of using poset pinball to construct a similarly computationally effective and convenient module bases for a GKM-compatible subspace by exploiting the structure of the ambient GKM space.

We briefly recall the setting of our results.  Let $N: \C^n \to \C^n$
be a nilpotent operator. Let
$h: \{1,2,\ldots, n\} \to \{1,2,\ldots,
n\}$ be a function satisfying $h(i) \geq i$ for all $1 \leq i \leq n$ and $h(i+1)
\geq h(i)$ for all $1 \leq i < n$. 
The associated Hessenberg variety $\Hess(N,h)$ is then 
defined as the following subvariety of
$\mathcal{F}\ell ags(\C^n)$: 
\[
\Hess(N,h) := \{ V_\bullet = (0 \subseteq V_1 \subseteq V_2 \subseteq \cdots
\subseteq V_{n-1} \subseteq V_n = \C^n) \hsm \mid \hsm  NV_i \subseteq
V_{h(i)} \textup{ for all } i = 1, \ldots, n \}. 
\]
Since we deal exclusively with type $A$ in this paper, henceforth we
omit this phrase from our terminology.  Two special cases of
Hessenberg varieties are of particular interest in this manuscript:
when $N$ is the principal nilpotent operator (in this case $\Hess(N,h)$ is
called a \textbf{regular nilpotent Hessenberg variety}) and when $h$
is the identity function $h(i)=i$ for all $1 \leq i \leq n$ (in this case
$\Hess(N,h)$ is called a \textbf{nilpotent Springer variety} and is
sometimes denoted $\mathcal{S}_N$).  Hessenberg varieties arise in many
areas of mathematics, including geometric representation theory
\cite{Spa76, Shi85, Fun03}, numerical analysis \cite{DeMProSha92},
mathematical physics \cite{Kos96, Rie03}, combinatorics \cite{Ful99},
and algebraic geometry \cite{BriCar04, CarrellKaveh:2008}, so it is of interest
to explicitly analyze their topology, e.g. the structure of their
(equivariant) cohomology rings. We do so through poset pinball and Schubert calculus techniques, as initiated and developed in \cite{HarTym09, HarTym10, BayHar10a} and briefly recalled below.

The following relationship between two group actions on the nilpotent Hessenberg variety and the flag variety respectively allows us to use the theory of GKM-compatible subspaces and poset pinball.
There is a natural $S^1$ subgroup
of the unitary diagonal matrices $T$ which acts on $\Hess(N,h)$ (defined precisely in
Section~\ref{sec:background}). The group $T$, the maximal torus of $U(n,\C)$, acts on $\Flags(\C^n)$ in the standard fashion. It turns out that the $S^1$-fixed points $\Hess(N,h)^{S^1}$ are a subset of the $T$-fixed points $\Flags(\C^n)^T \cong S_n$. 
Moreover, the inclusion of $\Hess(N,h)$
into $\Flags(\C^n)$ and the inclusion of groups $S^1$ into $T$ then induces a natural ring homomorphism
\begin{equation}\label{eq:intro-proj}
H^*_T(\mathcal{F}\ell ags(\C^n)) \to H^*_{S^1}(\Hess(N,h)).
\end{equation} 
As mentioned above, it is well-known in Schubert calculus that the equivariant Schubert classes $\{\sigma_w\}_{w \in S_n}$ are a computationally convenient $H^*_T(\pt)$-module basis for $H^*_T(\Flags(\C^n))$. 
We refer to 
the images in $H^*_{S^1}(\Hess(N,h))$ of the equivariant
Schubert classes $\{\sigma_w\}_{w \in S_n}$ via the projection~\eqref{eq:intro-proj} 
as \textbf{Hessenberg Schubert classes}. Given this setup and following \cite{HarTym10}, the game of poset pinball uses the data of  the fixed points $\Flags(\C^n)^T \cong S_n$ (considered as a partially ordered set with respect to Bruhat order) and the subset
\[
\Hess(N,h)^{S^1} \subseteq \Flags(\C^n)^T \cong S_n
\]
to determine a set of \textbf{rolldowns} in $S_n$. It is shown in \cite{HarTym10} that, under certain circumstances (one of which is discussed in more detail below), such a set of rolldowns in turn specifies a subset of the Hessenberg Schubert classes which form a $H^*_{S^1}(\pt)$-\textbf{module basis} of $H^*_{S^1}(\Hess(N,h))$. Thus poset pinball is an important tool for building computationally effective module bases for the equivariant cohomology of Hessenberg varieties. 
Indeed, the results of \cite{HarTym09} accomplish precisely this goal -- \emph{i.e.} of constructing a module basis via poset pinball techniques -- in the special case of the Peterson variety, which is the regular nilpotent Hessenberg variety with Hessenberg function $h$ defined by $h(i)=i+1$ for $1 \leq i \leq n-1$ and $h(n)=n$.  Exploiting this explicit module basis, in \cite[Theorem 6.12]{HarTym09} the second author and Tymoczko give a manifestly positive \textbf{Monk
  formula} for the product of a degree-$2$ Peterson Schubert class
with an arbitrary Peterson Schubert class, expressed as a
$H^*_{S^1}(\pt)$-linear combination of Peterson Schubert classes.
This is an example of equivariant
Schubert calculus in the realm of Hessenberg varieties, and it is an open problem to generalize the results of \cite{HarTym09} to a wider class of Hessenberg varieties. 

We now describe our main results.  First, we explain in detail an
algorithm which we dub the \textbf{dimension pair algorithm} and which
associates to each $S^1$-fixed point $w \in \Hess(N,h)^{S^1}$ a
permutation in $S_n$, which we call the \textbf{rolldown of $w$} following terminology in \cite{HarTym10} and denoted $\roll(w) \in S_n$.  In the special
cases of regular nilpotent Hessenberg varieties and nilpotent Springer
varieties, we show that the set $\{\roll(w)\}_{w \in
  \Hess(N,h)^{S^1}}$ can be interpreted as the result of a \textbf{successful game of 
  Betti pinball} (in the sense of \cite{HarTym10}). The
main motivation for our construction is that a successful outcome of Betti pinball can, under some circumstances, produce a module basis
for the associated equivariant cohomology ring (cf. \cite[Section
4.3]{HarTym10}). In this sense, our algorithm represents a significant
step towards the construction of module bases for the equivariant
cohomology rings of general nilpotent Hessenberg varieties,
thus extending the theory developed in \cite{HarTym09, HarTym10}. 
Although we formulate our
algorithm in terms of dimension pairs and permissible fillings
following terminology of 
Mbirika \cite{Mbirika:2010}, the essential idea
comes from a correspondence between Hessenberg affine cells and
certain Schubert polynomials which we learned from Erik Insko.

Second, for a specific case of a regular nilpotent Hessenberg
variety which we call a \textbf{$334$-type} Hessenberg
variety, we 
prove that the set of rolldowns $\{\roll(w)\}_{w \in \Hess(N,h)^{S^1}}$ obtained from the dimension pair algorithm is in fact
\textbf{poset-upper-triangular} in the sense of \cite{HarTym10}. As shown in \cite{HarTym10}, this is one of the possible circumstances under which we can conclude that the 
corresponding set of Hessenberg Schubert 
classes forms a module basis for the $S^1$-equivariant
cohomology ring of the variety. 
Thus our result gives rise to a new family of
examples of Hessenberg varieties (and GKM-compatible
subspaces) for which poset pinball successfully produces explicit
module bases. We mention that the dimension pair algorithm also produces
module bases in a special case of Springer varieties
\cite{DewHar10}. Although we do not know whether the dimension pair algorithm
always succeeds in producing module bases for the $S^1$-equivariant
cohomology rings for a general nilpotent Hessenberg variety, the 
evidence thus far is suggestive. 
We leave further investigation to future work.

We give a brief summary of the contents of this manuscript.  In
Section~\ref{sec:background} we recall some definitions and
constructions necessary for later statements. In
Section~\ref{subsec:dimension pair} we describe the
dimension pair algorithm
and prove that the result of the algorithm
satisfies the conditions to be the outcome of a successful game of
Betti poset pinball in the special cases of regular nilpotent
Hessenberg varieties and nilpotent Springer varieties. We briefly
review in Section~\ref{subsec:review pinball} the theory developed in
\cite{HarTym10} which show that, 
if the rolldown set obtained from a successful game of Betti poset pinball also satisfies
poset-upper-triangularity conditions, then it yields a module basis in equivariant
cohomology.  In Sections~\ref{sec:upper triangularity} and~\ref{sec:combinatorics} we prove that
the dimension pair algorithm produces a poset-upper-triangular module
basis in a special class of regular nilpotent Hessenberg varieties
which we call $334$-type Hessenberg varieties.  We close with some
open questions in
Section~\ref{sec:questions}.

\medskip
\noindent \textbf{Acknowledgements.} We thank Erik Insko for
explaining to us the correspondence between the elements of an affine
paving of regular nilpotent Hessenberg varieties and certain Schubert
polynomials which motivates our dimension 
pair algorithm. 
We
thank Barry Dewitt and Aba Mbirika for useful conversations and Rebecca Goldin for reviewing an initial draft of this manuscript and her many excellent suggestions for improving exposition. 
We are particularly indebted to Julianna
Tymoczko for her ongoing support, for answering many questions, and for her suggestions on an earlier draft of this paper. 

\section{Background}\label{sec:background}

We begin with necessary definitions and terminology for what
follows. In Section~\ref{subsec:hessenbergs} we recall the geometric
objects and the group actions under consideration. In
Section~\ref{subsec:permissible} we recall some combinatorial
definitions associated to Young diagrams. We recall a bijection
between Hessenberg fixed points and certain fillings of Young diagrams
in Section~\ref{subsec:bijection}. The discussion closely follows
previous work (e.g. \cite{HarTym09, HarTym10} and also \cite{Tym06})
so we keep exposition brief.

\subsection{Hessenberg varieties, highest forms, and fixed
  points}\label{subsec:hessenbergs}

By the \textbf{flag variety} we mean the homogeneous space
$GL(n,\C)/B$ which is also identified with 
\[
\mathcal{F}\ell ags(\C^n) := 
\{ V_{\bullet} = (\{0\} \subseteq V_1 \subseteq V_2 \subseteq \cdots V_{n-1} \subseteq
V_n = \C^n) \hsm \mid \hsm \dim_{\C}(V_i) = i \}. 
\]
A \textbf{Hessenberg function} is a function
$h: \{1,2,\ldots, n\} \to \{1,2,\ldots,
n\}$ satisfying $h(i) \geq i$ for all $1 \leq i \leq n$ and $h(i+1)
\geq h(i)$ for all $1 \leq i < n$. We frequently denote a Hessenberg
function by listing its values in sequence, $h = (h(1), h(2), \ldots,
h(n)=n)$. 
Let $N: \C^n \to \C^n$ be a linear operator. 
The \textbf{Hessenberg variety} 
$\Hess(N,h)$ is defined as the following subvariety of 
$\mathcal{F}\ell ags(\C^n)$: 
\begin{equation}\label{eq:def-Hess}
\Hess(N,h) := \{ V_{\bullet}  \in \mathcal{F}\ell ags(\C^n) \;
\vert \;  N V_i \subseteq
V_{h(i)} \text{ for all } i=1,\ldots,n\} \subseteq \mathcal{F}\ell
ags(\C^n). 
\end{equation}
If $N$ is nilpotent, we say $\Hess(N,h)$ is a \textbf{nilpotent
  Hessenberg variety}, and if $N$ is the principal nilpotent operator
(i.e. has one Jordan block with eigenvalue $0$), then $\Hess(N,h)$ is
called a \textbf{regular nilpotent Hessenberg variety}. If $N$ is
nilpotent and $h$ is the
identity function $h(i)=i$ for all $1 \leq i \leq n$ then $\Hess(N,h)$
is called a \textbf{nilpotent Springer variety} and often denoted
$\mathcal{S}_N$.  In this manuscript we study in some detail the regular
nilpotent case, and as such sometimes notate $\Hess(N,h)$ as
$\Hess(h)$ when $N$ is understood to be the standard principal
nilpotent operator.

Suppose given $N$ a nilpotent matrix in standard Jordan canonical
form. It turns out that for many of our statements below we must use a 
choice of conjugate of $N$ which is in \textbf{highest
  form} \cite[Definition 4.2]{Tym06}. We recall the following. 

\begin{definition} \textbf{(\cite[Definition 4.1 and Definition
    4.2]{Tym06})}
\begin{itemize} 
\item Let $X$ be any $m \times n$ matrix . We call the entry $X_{ik}$ a 
\textbf{pivot} of $X$
if $X_{ik}$ is nonzero 
and if all entries below and to its left vanish, i.e.,
\(X_{ij} = 0\) if \(j < k\) and \(X_{jk} = 0\) if \(j > i.\) 
Moreover, given $i$, define $r_i$ to be the row of $X_{r_i,i}$ if the
entry is a pivot, and $0$ otherwise.

\item Let $N$ be an upper-triangular nilpotent $n \times n$ matrix.
Then we say $N$ is in \textbf{highest 
form} if its pivots form a nondecreasing sequence, namely 
\(r_1 \leq r_2 \leq \cdots \leq r_n.\) 
\end{itemize}
\end{definition}

We do not require the details of the theory of highest forms of linear
operators; for the purposes of the present manuscript
it suffices to remark firstly that when
$N$ is the principal nilpotent matrix then $N$ is already in highest
form, and secondly that any nilpotent matrix can be conjugated by an
appropriate $n \times n$ permutation matrix $\sigma$ so that $N_{hf} :=
\sigma N\sigma^{-1}$ is in highest form. However the following observation will be
relevant in Section~\ref{subsec:bijection}.

\begin{remark}\label{remark:choice of highest form} 
  In this manuscript we always assume that our highest form $N_{hf} =
  \sigma N \sigma^{-1}$
  has been chosen in accordance to the recipe described by Tymoczko in
  \cite[Section 4]{Tym06}. Since the precise method of this
  construction is not relevant for the rest of the present manuscript
  we omit further explanation here. In the
  case when $N$ is principal nilpotent we take $N_{hf} = N$ since $N$ is
  already in highest form and this is the form chosen by Tymoczko in
  \cite{Tym06}.  A more detailed discussion of highest forms as it
  pertains to poset pinball theory is in \cite{DewHar10}.
\end{remark}

For details on the following facts we refer the reader to
e.g. \cite{HarTym09, HarTym10, Tym06} and references therein. 
Let $N$ be an $n \times n$ nilpotent matrix in Jordan
canonical form and let $\sigma$ denote a permutation matrix such
that $N_{hf}:=\sigma N\sigma^{-1}$ is in highest form. 
It is known and straightforward to show that the following $S^1$ 
subgroup of $U(n,\C)$ preserves $\Hess(N,h)$ for $N$ as above and
any Hessenberg function $h$: 
\begin{equation}\label{eq:def-circle}
S^1  =   \left\{ \left. \begin{bmatrix} t^n & 0 & \cdots & 0 \\ 0 & t^{n-1} &  &
      0 \\ 0 & 0 & \ddots & 0 \\ 0 & 0 &  & t \end{bmatrix}
  \; \right\rvert \;  t \in \C, \; \|t\| = 1 \right\}  \subseteq T^n
\subseteq U(n,\C). 
\end{equation}
Here $T^n$ is the standard maximal torus of
$U(n,\C)$ consisting of diagonal unitary matrices.  

This implies that the conjugate circle subgroup $\sigma S^1 \sigma^{-1}$
preserves $\Hess(N_{hf},h)$. By abuse of notation we will
denote both circle subgroups by $S^1$, since it is clear by context
which is meant.  The $S^1$-fixed points of $\Hess(N,h)$ and
$\Hess(N_{hf}, h)$ are isolated, and are a subset of the $T^n$-fixed
points of $\Flags(\C^n)$. Since the set of
$T^n$-fixed points $\Flags(\C^n)^{T^n}$ may be identified with the
Weyl group $W = S_n$, and since $\Hess(N, h)^{S^1}$ (respectively
$\Hess(N_{hf}, h)^{S^1}$) is a subset of $\Flags(\C^n)^{T^n}$, any
Hessenberg fixed point may be thought of as a permutation \(w \in S_n.\)

\subsection{Permissible fillings, dimension pairs, lists of top 
parts, and associated permutations}\label{subsec:permissible}

Recall that there is a bijective correspondence between the set of conjugacy
classes of nilpotent $n \times n$ complex matrices $N$ and Young
diagrams\footnote{We use English notation for Young diagrams.} with $n$ boxes, given by associating to $N$ the Young diagram
$\lambda$ with row lengths the sizes of the Jordan blocks of
$N$ listed in weakly decreasing order.  We will use this bijection to
often treat such $N$ and $\lambda$ as the same data; we sometimes
denote by $\lambda_N$ the Young diagram given as above corresponding
to a nilpotent $N$.

For more details on the following see \cite{Mbirika:2010}.

\begin{definition}\label{definition:permissible filling}
  Let $\lambda$ be a Young diagram with $n$ boxes. Let $h: \{1,2,
  \ldots, n\} \to \{1,2,\ldots, n\}$ be a Hessenberg function. A
  \textbf{filling} of $\lambda$ by the alphabet $\{1,2,\ldots,n\}$ 
is an injective placing of the
  integers $\{1,2,\ldots, n\}$ into the boxes of $\lambda$.  A filling
  of $\lambda$ is called a \textbf{$\mathbf{(h,\lambda)}$-permissible
    filling} if for every horizontal adjacency \(\begin{array}{|c|c|}
    \cline{1-2} k & j \\ \cline{1-2} \end{array}\) in the filling we have \(k \leq
  h(j).\) 
\end{definition}

\begin{remark}
In this manuscript the $\lambda$ and
$h$ will frequently be understood by context. When there is no danger
of confusion we simply refer to \textbf{permissible fillings}. 
\end{remark}

\begin{example}\label{example:permissible filling}
Let $n=5$. Suppose $\lambda = (5)$ 
and $h=(3,3,4,5,5)$. Then
\(\begin{array}{|c|c|c|c|c|} \cline{1-5} 2 & 4 &  3 & 1 &5 \\
  \cline{1-5} \end{array}\) is a permissible filling, whereas 
\(\begin{array}{|c|c|c|c|c|} \cline{1-5} 2 & 3 & 4 & 1 &5 \\
  \cline{1-5} \end{array}\) is not, since $4 \not \leq h(1)$. 
\end{example}

We denote a permissible filling of $\lambda$ by $T$, in analogy with
standard notation for Young tableaux. 
Next we focus attention on certain pairs of entries in a permissible filling
$T$.

\begin{definition}\label{definition:dimension pair}
  Let $h: \{1,2,\ldots,n\} \to \{1,2,\ldots,n\}$ be a Hessenberg
  function and $\lambda$ a Young diagram with $n$ boxes.  A pair
  $(a, b)$ is a \textbf{dimension pair} of an
  $(h,\lambda)$-permissible filling $T$ of $\lambda$ if the following
  conditions hold:
\begin{enumerate} 
\item \(b > a, \)
\item $b$ is either 
\begin{itemize}
\item below $a$ in the same column of $a$, or 
\item anywhere in a column strictly to the left of the column of $a$, 
\end{itemize} 
and 
\item if there exists a box with filling $c$ directly adjacent to the right
of $a$, then $b \leq h(c)$. 
\end{enumerate}
For a dimension pair $(a,b)$ of $T$, we will refer to $b$ as the
\textbf{top part} of the dimension pair. 
\end{definition}

\begin{example}\label{example:dimension pair} 
Let $\lambda, h$ be as in Example~\ref{example:permissible
  filling}. The 
dimension pairs in the permissible filling
\(\begin{array}{|c|c|c|c|c|} \cline{1-5} 2 & 4 &  3 & 1 &5 \\
  \cline{1-5} \end{array}\)
are $(1,2), (1,3)$, and $(1,4)$. Note that $(3,4)$ is not a dimension
pair because $1$ is directly to the right of the $3$ and $4 \not \leq
h(1)$. 
\end{example}

Given a permissible filling $T$ of $\lambda$, we follow \cite{Mbirika:2010}
and denote by $DP^T$ the set of dimension pairs of $T$. 
For each integer $\ell$ with $2 \leq \ell \leq n$, let 
\begin{equation}\label{eq:def xell}
x_\ell := \lvert \{ (a,\ell) \hsm \vert \hsm (a,\ell) \in DP^T \} \rvert
\end{equation}
so $x_\ell$ is the number of times $\ell$ occurs as a top part in
the set of dimension pairs of $T$. From the definitions
it follows that \(0 \leq x_\ell \leq \ell-1 \textup{ for all } 2 \leq \ell \leq n.\)
We call the integral vector 
$\mathbf{x} = (x_2, x_3, \ldots, x_n)$ the \textbf{list of top parts}
of $T$.

To each such $\mathbf{x}$ we associate a permutation in $S_n$ as
follows. As a preliminary step, for each $\ell$ with $2 \leq \ell \leq n$ define 
\[
u_{\ell}(\mathbf{x}) := \begin{cases} s_{\ell-1} s_{\ell-2} \cdots s_{\ell-x_\ell}
  \quad \textup{ if } x_\ell > 0 \\
1 \quad \textup{ if } x_\ell = 0 
\end{cases} 
\]
where $s_i$ denotes the simple transposition $(i,i+1)$ in $S_n$ and
$1$ denotes the identity permutation. Now
define the association 
\begin{equation}\label{eq:permutation from top parts} 
\mathbf{x} \mapsto  \omega(\mathbf{x}) := u_2(\mathbf{x}) u_3(\mathbf{x})
\cdots u_n(\mathbf{x}) \in S_n.
\end{equation}
It is not difficult to see that~\eqref{eq:permutation from top parts}
is a bijection between the set of integral vectors $\mathbf{x} \in
\Z^{n-1}$ satisfying $0 \leq x_{\ell} \leq \ell-1$ for all $2 \leq
\ell \leq n-1$ and the group $S_n$. In fact the word given
by~\eqref{eq:permutation from top parts} is a reduced word
decomposition of $\omega(\mathbf{x})$ and the $x_\ell$ count the number of
inversions in $\omega(\mathbf{x})$ with $\ell$ as the higher integer. 
The following simple fact will be used later. 

\begin{fact}\label{fact:order}
  Suppose $\mathbf{x} = (x_2, \ldots, x_n), \mathbf{y} = (y_2,
  \ldots,y_n) \in \Z^{n-1}_{\geq 0}$ are both lists of top
  parts. Suppose further that for all $2 \leq \ell \leq n$, we have
  $x_\ell \leq y_\ell$. Then $\omega(\mathbf{x}) \leq
  \omega(\mathbf{y})$ in Bruhat order. This follows immediately from
  the definition~\eqref{eq:permutation from top parts}. 
\end{fact}

\begin{example}
Continuing with Examples~\ref{example:permissible filling}
and~\ref{example:dimension pair}, for the permissible filling \(\begin{array}{|c|c|c|c|c|} \cline{1-5} 2 & 4 &  3 & 1 &5 \\
  \cline{1-5} \end{array}\) the set $DP^T$ of top parts of dimension
pairs is $\{2,3,4\}$, yielding the integer vector $\mathbf{x} =
(1,1,1,0)$. The associated permutation $\omega(\mathbf{x})$ is then \(s_1 s_2 s_3.\) 
\end{example}

\begin{example}\label{example:associated perm}
  Let $\lambda, h$ be as in Example~\ref{example:permissible
    filling}. The filling \(\begin{array}{|c|c|c|c|c|} \cline{1-5} 4 & 3 &  2 & 1 &5 \\
  \cline{1-5} \end{array}\) is also permissible, with dimension pairs
$(1,2),(1,3),(1,4),(2,3)$. Hence $\mathbf{x} = (1,2,1,0)$ 
and the associated permutation $\omega(\mathbf{x})$ is $s_1 (s_2 s_1) s_3$. 
\end{example}

\subsection{Bijection between fixed points and permissible fillings}\label{subsec:bijection}

For nilpotent Hessenberg varieties, the $S^1$-fixed points
$\Hess(N,h)^{S^1}$ are in bijective correspondence with the set of
permissible fillings of the Young diagram $\lambda = \lambda_N$, as we
now describe.  We
will use this correspondence in the formulation of our dimension pair
algorithm.

Suppose $\lambda$ is a Young diagram with $n$ boxes. We begin by defining a bijective correspondence between the set
$\Fill(\lambda)$ of \emph{all} fillings (not necessarily permissible)
of $\lambda$ with permutations in $S_n$. Given a filling, read the
entries of the filling by reading along each column from the bottom to
the top, starting with the leftmost column and proceeding to the
rightmost column. The association $\Fill(\lambda) \leftrightarrow S_n$
is then given by interpreting the resulting word as the one-line notation
of a permutation. For example the filling
\[
T = {\def\lr#1{\multicolumn{1}{|@{\hspace{.6ex}}c@{\hspace{.6ex}}|}{\raisebox{-.3ex}{$#1$}}}
\raisebox{-.6ex}{$\begin{array}[b]{ccc}
\cline{1-1}\cline{2-2}\cline{3-3}
\lr{1}&\lr{2}&\lr{3}\\
\cline{1-1}\cline{2-2}\cline{3-3}
\lr{4}&\lr{5}\\
\cline{1-1}\cline{2-2}
\lr{6}\\
\cline{1-1}
\end{array}$}
}
\]
has associated permutation $641523$. It is easily seen that this is a
bijective corresondence. Given a filling $T$ of $\lambda$ we denote
its associated permutation by $\phi_\lambda(T)$.

\begin{remark}\label{remark:regular case}
  In the case when $N$ is the principal nilpotent $n \times n$ matrix,
  the corresponding Young diagram $\lambda = \lambda_N = (n)$ has only
  one row, so the above correspondence simply reads off the (one row
  of the) filling from left to right. In this case we abuse
  notation and denote $\phi_\lambda^{-1}(w)$ by just $w$. For instance, the
  permissible filling of $\lambda=(5)$ in
  Example~\ref{example:associated perm} has associated permutation
  $43215$.
\end{remark}

Now let 
\begin{equation}\label{eq:def PFill} 
\PFill(\lambda, h)
\end{equation}
denote the set of $(h,\lambda)$-permissible fillings of
$\lambda$. Recall that elements in $\Hess(N,h)^{S^1}$ are viewed as
permutations in $S_n$ via the identification $\Flags(\C^n)^{T^n} \cong
S_n$.  The next proposition follows from the definitions and some
linear algebra. It is proven and discussed in more detail in
\cite{DewHar10}, where the notation used is slightly different.

\begin{proposition}\label{proposition:fixed points and fillings}
  Fix $n$ a positive integer.  Let $h: \{1,2,\ldots, n\} \to
  \{1,2,\ldots, n\}$ be a Hessenberg function and $\lambda$ a Young
  diagram with $n$ boxes. Suppose $N_{hf}$ is a nilpotent
  operator in highest form as chosen in \cite{Tym06}
  (cf. Remark~\ref{remark:choice of highest form}) with
  $\lambda_{N_{hf}} = \lambda$. Let
  $\Hess(N_{hf}, h)$ denote
  the associated nilpotent Hessenberg variety. Then the map
  from the $S^1$-fixed points $\Hess(N_{hf}, h)^{S^1}$ to the
  set of permissible fillings $\PFill(\lambda, h)$ 
\begin{equation}\label{eq:def PF_w}
w  \in \Hess(h)^{S^1} \subseteq S_n \mapsto \phi_{\lambda}^{-1}(w^{-1}) \in \PFill(\lambda,h)
\end{equation}
is well-defined and is a bijection. 
\end{proposition}

\begin{remark}\label{remark:PF_w in case N principal}
  In the case when $N$ is the principal nilpotent $n \times n$ matrix,
  $\lambda$ is the Young diagram with only one row. Thus the
  map~\eqref{eq:def PF_w} above simplifies to $w \mapsto w^{-1}$ where
  we abuse notation (cf. Remark~\ref{remark:regular case}) and denote
  $\phi_{\lambda}^{-1}(w^{-1})$ by $w^{-1}$. 
\end{remark}

\section{The dimension pair algorithm for Betti poset
  pinball for nilpotent Hessenberg varieties}\label{sec:aba}

In this section we first explain the \textbf{dimension pair
  algorithm} which associates to any Hessenberg fixed point a
permutation in $S_n$.  The name is due to the fact that the
construction proceeds by computing dimension pairs in appropriate
permissible fillings. We then interpret this algorithm as a method
for choosing \textbf{rolldowns} associated to the Hessenberg fixed
points in a game of \textbf{Betti poset pinball} in the
sense of \cite{HarTym10}.  The algorithm makes sense for any nilpotent
Hessenberg variety, so it is defined in that generality in
Section~\ref{subsec:dimension pair}. However, our proof that the
algorithm produces a successful outcome of Betti poset pinball in
the sense of \cite{HarTym10} is only for the special cases of regular
nilpotent Hessenberg varieties and nilpotent Springer varieties.  In
Section~\ref{subsec:review pinball} we briefly recall the setup and
necessary results of poset pinball which allow us to conclude that our
poset pinball result yields an explicit module basis for
equivariant cohomology.

\subsection{The dimension pair algorithm for nilpotent Hessenberg
  varieties}\label{subsec:dimension pair}

Let $N_{hf}$ be a nilpotent $n \times n$
matrix in highest form chosen as in Remark~\ref{remark:choice of
  highest form} and let $\lambda := \lambda_{N_{hf}}$. 
Let $h: \{1,2, \ldots, n\} \to \{1,2,\ldots, n\}$
be a Hessenberg function and $\Hess(N_{hf},h)$ the corresponding
nilpotent Hessenberg variety.

The definition of the dimension pair algorithm is pure 
combinatorics. It produces for each Hessenberg fixed point $w \in
\Hess(N_{hf}, h)^{S^1}$ an element in $S_n$. Following terminology of poset
pinball, we denote this function by 
\[
\roll: \Hess(N_{hf}, h)^{S^1} \to S_n.
\]

\begin{definition} \textbf{(``The dimension pair algorithm'')} 
We define $\roll: \Hess(N_{hf}, h)^{S^1} \to S_n$ as follows:  
\begin{enumerate}
\item Let $w \in \Hess(N_{hf}, h)^{S^1}$ and let
  $\phi_{\lambda}^{-1}(w^{-1})$ be its corresponding permissible filling as
  defined in~\eqref{eq:def PF_w}.
\item Let $DP^{\phi_{\lambda}^{-1}(w^{-1})}$ be the set of dimension pairs in the permissible
  filling $\phi_{\lambda}^{-1}(w^{-1})$. 
\item For each $\ell$ with $2 \leq \ell \leq n$, 
set 
\[
x_\ell := \lvert \{  (a,\ell) \hsm \vert \hsm (a, \ell) \in DP^{\phi_{\lambda}^{-1}(w^{-1})} \} \rvert
\]
as in~\eqref{eq:def xell} and define \(\mathbf{x} := (x_2, \ldots, x_n).\) 
\item Define $\roll(w) := (\omega(\mathbf{x}))^{-1}$ where $\omega(\mathbf{x})$
  is the permutation associated to the integer vector $\mathbf{x}$
  defined in~\eqref{eq:permutation from top parts}.
\end{enumerate}
\end{definition}

\begin{example}
  Let $\lambda, h$ be as in Example~\ref{example:permissible
    filling}. The permutation $w = 43215 \in S_n$
  is in $\Hess(N_{hf}, h)^{S^1}$, as can be
  checked. The associated
  permissible filling is \(\begin{array}{|c|c|c|c|c|} \cline{1-5} 4 & 3 &  2 & 1 &5 \\
  \cline{1-5} \end{array}.\) In Example~\ref{example:associated
  perm} we saw that the associated permutation is $s_1 (s_2 s_1)
s_3$, so we conclude $\roll(w) = s_3 (s_1 s_2) s_1$. 
\end{example}

We next show that the rolldown function $\roll: \Hess(h)^{S^1} \to S_n$
defined by the dimension pair algorithm above satisfies the conditions
to be a \textbf{successful outcome of Betti poset pinball} as in
\cite{HarTym10} in certain cases of nilpotent Hessenberg varieties. 
The statement of one of the conditions requires advance knowledge of the
Betti numbers of nilpotent Hessenberg varieties, for which we recall the
following result (reformulated in our language) from \cite{Tym06}.  

\begin{theorem}\label{theorem:paving} \textbf{(\cite[Theorem
    1.1]{Tym06})}
Let 
$N_{hf}: \C^n \to \C^n$ be a nilpotent matrix 
in highest form chosen as in Remark~\ref{remark:choice of highest
  form} and let $\lambda := \lambda_{N_{hf}}$. Let $h:
\{1,2,\ldots,n\} \to \{1,2,\ldots,n\}$ be a Hessenberg function and
let $\Hess(N_{hf}, h)$ denote the corresponding nilpotent Hessenberg
variety. 
There is a paving by (complex) affine cells of $\Hess(N_{hf},h)$ 
such that: 
  \begin{itemize}
  \item the affine cells are in one-to-one correspondence with
    $\Hess(N_{hf},h)^{S^1}$, and
 \item the (complex) dimension of the affine cell $C_w$ corresponding to a
    fixed point \(w \in \Hess(N,h)^{S^1}\) is 
\begin{equation}\label{eq:dim C_w}
\dim_{\C}(C_w) =  \lvert DP^{\phi_{\lambda}^{-1}(w^{-1})} \rvert.
\end{equation}
\end{itemize}
\end{theorem}

In particular, Theorem~\ref{theorem:paving} implies that the odd Betti
numbers of $\Hess(N_{hf},h)$ are $0$, and the $2k$-th even Betti
number is precisely the number of fixed points $w$ in
$\Hess(N_{hf},h)^{S^1}$ such that $\lvert DP^{\phi_{\lambda}^{-1}(w^{-1})}
\rvert = k$.  Given the regular nilpotent Hessenberg variety
$\Hess(N_{hf}, h)$, denote by $b_k$ its $2k$-th Betti number, i.e.
\[
b_k := \dim_\C H^{2k}(\Hess(N_{hf},h)).
\]
We may now formulate the conditions that guarantee that $\roll:
\Hess(N_{hf}, h)^{S^1} \to S_n$ is a successful outcome of Betti
pinball. For more details we refer the reader to 
\cite[Section 3]{HarTym10}. It suffices to check the
following: 

\begin{enumerate}
\item $\roll: \Hess(N_{hf},h)^{S^1} \to S_n$ is injective, 
\item for every $w \in \Hess(N_{hf}, h)^{S^1}$, we have $\roll(w) \leq w$ in
  Bruhat order, and 
\item for every $k \geq 0, k \in \Z$, we have 
\[
b_k = \left\lvert \left\{\roll(w) \hsm \vert \hsm  w \in \Hess(N_{hf},
    h)^{S^1} \textup{ with }
\ell(\roll(w)) = k \right\} \right\rvert
\]
where $\ell(\roll(w))$ denotes the Bruhat length of $\roll(w) \in
S_n$. 
\end{enumerate}

We prove each claim in turn. For the first assertion we restrict to
two special cases of Hessenberg varieties. 

\begin{lemma}\label{lemma:injective} 
Suppose that $\Hess(N_{hf},h)$ is either a regular nilpotent
Hessenberg variety or a nilpotent Springer variety. 
Then the function $\roll: \Hess(N_{hf}, h)^{S^1} \to S_n$ is injective. 
\end{lemma}

\begin{proof}
Since the association $\mathbf{x} = (x_2, x_3,
  \ldots, x_n) \mapsto \omega(\mathbf{x})$ 
 given in~\eqref{eq:permutation from top parts} is a
  bijection it suffices to show that
  the map which sends a Hessenberg fixed point $w
  \in \Hess(h)^{S^1}$ to the list of top parts $\mathbf{x}$ of its
  associated 
  permissible filling is injective. Mbirika shows that, in the cases of
  regular nilpotent Hessenberg varieties and nilpotent Springer varieties, 
  there exists an inverse to this map 
 (Mbirika works with monomials in $n-1$
  variables constructed from the list of top parts, but this is equivalent
  data) \cite[Section 3.2]{Mbirika:2010}. The result follows.
\end{proof}

\begin{lemma}\label{lemma:rolldowns bruhat less}
  For every $w \in \Hess(h)^{S^1}$, we have $\roll(w) \leq w$ in
  Bruhat order. 
\end{lemma}

\begin{proof}
Since Bruhat order is preserved under taking inverses, 
it suffices to prove that $\omega(\mathbf{x})$ is Bruhat-less than 
$w^{-1}$. 
For any 
permutation $u \in S_n$, set
\[
y_\ell := \{(a,\ell) \hsm \vert \hsm (a,\ell) \textup{ is an inversion
  in } u\}
\]
and let \(\mathbf{y} := (y_2, y_3, \ldots, y_n).\) Then the
association~\eqref{eq:permutation from top parts} applied to the
vector $\mathbf{y}$ recovers the permutation $u$. 
By definition of $\phi_{\lambda}$ and the definition of dimension
pairs, 
the set $DP^{\phi_{\lambda}^{-1}(w^{-1})}$ 
is always a subset of the set of inversions of the
permutation $w^{-1}$. 
From Fact~\ref{fact:order} 
it follows that the permutation $\omega(\mathbf{x})$ 
is Bruhat-less than $w^{-1}$ as desired. 
\end{proof}

\begin{lemma}\label{lemma:betti acceptable} 
Let 
$N_{hf}: \C^n \to \C^n$ be a nilpotent matrix 
in highest form chosen as in Remark~\ref{remark:choice of highest
  form} and let $\lambda := \lambda_{N_{hf}}$. 
Let $h: \{1,2,\ldots,n\} \to \{1,2,\ldots, n\}$ be a Hessenberg
function and $\Hess(N_{hf}, h)$ the associated nilpotent Hessenberg
variety. 
For every $k \geq 0, k \in \Z$, we have 
\[
b_k = \left\lvert \left\{\roll(w) \hsm \vert \hsm  w \in \Hess(h)^{S^1} \textup{ with }
\ell(\roll(w)) = k \right\} \right\rvert,
\]
where $\ell(\roll(w))$ denotes the Bruhat length of $\roll(w) \in
S_n$. 
\end{lemma}

\begin{proof}
  By construction, $\roll(w)$ has a reduced word decomposition
  consisting of precisely $\lvert DP^{\phi_{\lambda}^{-1}(w^{-1})}
  \rvert$ simple transpositions. Hence its Bruhat length is $\lvert
  DP^{\phi_{\lambda}^{-1}(w^{-1})} \rvert$. By
  Theorem~\ref{theorem:paving}, $b_k$ is precisely the number of fixed
  points $w$ with $\lvert DP^{\phi_{\lambda}^{-1}(w^{-1})} \rvert =
  k$ so the result follows. 
\end{proof}

The following is immediate from the above lemmas and the
definition of Betti pinball given in \cite[Section 3]{HarTym10}.

\begin{proposition}\label{proposition:betti acceptable}
Suppose that $\Hess(N_{hf},h)$ is either a regular nilpotent
Hessenberg variety or a nilpotent Springer variety. 
Then the association $w \mapsto \roll(w)$ given by the dimension pair
  algorithm is a possible outcome of a successful game of Betti poset
  pinball played with ambient partially ordered set $S_n$ equipped
  with Bruhat order, rank function $\rho = \ell: S_n \to \Z$ given by
  Bruhat length, initial subset $\Hess(h)^{S^1} \subseteq S_n$, and
  target Betti numbers $b_k := \dim_\C H^{2k}(\Hess(h);\C)$.
\end{proposition}

\begin{remark}\label{remark:injective}
Lemmas~\ref{lemma:rolldowns bruhat less} and~\ref{lemma:betti
  acceptable} hold for general nilpotent $N_{hf}$ and Hessenberg
functions $h$. Hence to prove that Proposition~\ref{proposition:betti
  acceptable} holds for more general cases of nilpotent Hessenberg
varieties, it suffices to check that the injectivity assertion (1)
above 
holds. We do not know counterexamples where the injectivity 
fails. It would be of interest to clarify
the situation for more general $N_{hf}$ and $h$. 
\end{remark}

\subsection{Betti pinball, poset-upper-triangularity, and module
  bases}\label{subsec:review pinball}

In the context of a
\textbf{GKM-compatible subspace of a GKM space} \cite[Definition
4.5]{HarTym10}, it is explained in \cite[Section 4]{HarTym10} that the
outcome of a game of poset pinball may be interpreted as specifying a
set of equivariant cohomology classes which, under additional
conditions, yields a module basis for the equivariant cohomology of
the GKM-compatible subspace.  In this paper, the GKM space is the flag
variety $\Flags(\C^n)$ with the standard $T^n$-action and the
GKM-compatible subspace is $\Hess(N_{hf},h)$ with the $S^1$-action
specified above. Consider the $H^*_{T^n}(\pt)$-module basis for
$H^*_{T^n}(\Flags(\C^n))$ given by the equivariant Schubert
classes $\{\sigma_w\}_{w \in S_n}$. The dimension pair algorithm then
specifies the set
\[
\{ p_{\roll(w)} \hsm \vert \hsm w \in \Hess(N_{hf},h)^{S^1} \} \subseteq
H^*_{S^1}(\Hess(N_{hf},h))
\]
where for any $u \in S_n$ the class $p_u := \pi(\sigma_u)$ is defined to be the image
of $\sigma_u$ under the natural projection map 
\[
\pi: H^*_{T^n}(\Flags(\C^n)) \to H^*_{S^1}(\Hess(N_{hf},h))
\]
induced by the inclusion of groups $S^1 \into T^n$ and
the $S^1$-equivariant inclusion of
spaces $\Hess(N_{hf},h) \into \Flags(\C^n)$. We refer to the images
$p_u$ as \textbf{Hessenberg Schubert classes}.

Following the methods of
\cite{HarTym10} we view $H^*_{T^n}(\Flags(\C^n))$ and
$H^*_{S^1}(\Hess(N_{hf},h))$ as subrings of 
\[
H^*_{T^n}((\Flags(\C^n))^{T^n}) \cong \bigoplus_{w \in S_n}
H^*_{T^n}(\pt) 
\quad \textup{respectively} \quad
H^*_{S^1}((\Hess(N_{hf},h))^{S^1}) \cong \bigoplus_{w \in \Hess(N_{hf},h)^{S^1}}
H^*_{S^1}(\pt).
\]
We denote by $\sigma_w(w'), p_{\roll(w)}(w')$ the value of
the $w'$-th coordinate in the direct sums above, for $w, w' \in S_n$
or $w, w' \in \Hess(N_{hf},h)^{S^1}$ respectively.  If
\begin{equation}\label{eq:p upper triangular}
p_{\roll(w)}(w) \neq 0, \quad \textup{ and } \quad p_{\roll(w)}(w') = 0
\textup{ if } w \not \leq w' 
\end{equation}
for all $w, w' \in \Hess(N_{hf},h)^{S^1}$ then 
the set $\{p_{\roll(w)} \hsm \vert \hsm w \in \Hess(N_{hf},h)^{S^1} \}$ in $H^*_{S^1}(\Hess(N_{hf},h))$ is
called \textbf{poset-upper-triangular} (with respect to the partial order on
$\Hess(N_{hf},h)^{S^1} \subseteq S_n$ induced from Bruhat order) 
\cite[Definition 2.3]{HarTym10}. Finally, recall that the cohomology
degree of
an equivariant Schubert class $\sigma_w$ (and hence also the
corresponding Hessenberg Schubert
class $p_{w}$) is $2 \cdot \ell(w)$.

The following is immediate from \cite[Proposition 4.14]{HarTym10} and
the above discussion. 

\begin{proposition}\label{proposition: if triangular then basis}
  Let $\Hess(N_{hf},h)$ be either a regular nilpotent Hessenberg
  variety or a nilpotent Springer variety. Let $\roll:
  \Hess(N_{hf},h)^{S^1} \to S_n$ be the dimension pair algorithm
  defined above. Suppose~\eqref{eq:p upper triangular} holds for all
  $w \in \Hess(N_{hf},h)^{S^1}$. Then the Hessenberg Schubert classes
  $\{p_{\roll(w)} \hsm \vert \hsm w \in \Hess(N_{hf},h)^{S^1}\}$ form
  a $H^*_{S^1}(\pt)$-module basis for the $S^1$-equivariant cohomology
  ring $H^*_{S^1}(\Hess(N_{hf},h))$.
\end{proposition}

Therefore, in order to prove that the Hessenberg Schubert classes
above form a module basis as desired, it suffices to show that they
satisfy the upper-triangularity conditions~\eqref{eq:p upper
  triangular} for all $w, w' \in \Hess(N_{hf},h)^{S^1}$. The proof of this assertion, for a special class of
regular nilpotent Hessenberg varieties closely related to Peterson
varieties, is the content of 
Sections~\ref{sec:upper triangularity} and~\ref{sec:combinatorics}.

We close the section with a brief discussion of matchings. Following
\cite[Section 4.3]{HarTym10}, define 
\[
\deg_{\Hess(N_{hf},h)}(w) := \dim_\C(C_w)
\]
to be the (complex) dimension of the affine cell $C_w$ containing 
the fixed point $w$ in Tymoczko's paving by affines of $\Hess(N_{hf},h)$ in
Theorem~\ref{theorem:paving}. Then from the discussion above we know 
\[
\deg_{\Hess(N_{hf},h)}(w) = \lvert DP^{\phi_{\lambda}^{-1}(w^{-1})} \rvert = \ell(\roll(w)),
\]
and since the cohomology degree of $p_{\roll(w)}$ is $2 \cdot
\ell(\roll(w))$, we see that the association $w \mapsto \roll(w)$ from
$\Hess(N_{hf},h)^{S^1} \to S_n$ is also a \textbf{matching} in the sense of
\cite{HarTym10} with respect to $\deg_{\Hess(N_{hf},h)}$ and rank function
$\rho$ on $S_n$ given by Bruhat length. Thus the fact that the
$\{p_{\roll(w)} \hsm \vert \hsm w \in \Hess(N_{hf},h)^{S^1} \}$ form a module
basis can also be deduced from \cite[Theorem 4.18]{HarTym10}.

\section{Poset-upper-triangularity of rolldown classes 
for $334$-type Hessenberg varieties}\label{sec:upper triangularity}

In this section and in Section~\ref{sec:combinatorics}
we analyze in detail the dimension pair algorithm in the case
of a Hessenberg variety which is 
closely related to the Peterson variety, and in particular prove that
the algorithm produces a poset-upper-triangular module basis for its
$S^1$-equivariant cohomology ring.
Here and below the nilpotent operator $N$ under consideration is always the principal
nilpotent, so we omit the $N$ from the notation and write
$\Hess(h)$. Similarly the corresponding Young diagram is always
$\lambda = (n)$ so we omit the $\lambda$ from notation and write
$\PFill(h)$ instead of $\PFill(\lambda, h)$. 

We fix for this discussion 
the Hessenberg function given by 
\begin{equation}\label{eq:334 Hessenberg} 
h(1)=h(2)=3, \quad h(i) = i+1 \textup{ for } 3 \leq i \leq n-1, \quad
\textup{ and } 
h(n)= n.
\end{equation}
The only difference between this function $h$ and the
Hessenberg function for the Peterson variety studied in
\cite{HarTym09} is that the value of $h(1)$ is $3$ instead of $2$. In
this sense this $h$ is ``close'' to the Peterson case.
Thus it is natural that much of our analysis follows that 
for Peterson varieties in \cite{HarTym09}, although it is still
necessary to 
introduce new ideas and terminology to handle the Hessenberg fixed
points in $\Hess(h)^{S^1}$ which do not arise in the Peterson case. 

The Hessenberg function $h$ in~\eqref{eq:334 Hessenberg} is
trivial if $n=3$ since in that case $h(1)=h(2)=h(3)=3$ which implies
that the corresponding Hessenberv variety $\Hess(h)$ is equal to the
full flag variety $\mathcal{F}\ell ags(\C^3)$. Hence we assume $n \geq
4$ throughout.  Under this assumption and following the notation
introduced in Section~\ref{sec:background}, the Hessenberg function is
of the form $h=(3,3,4, \cdots)$. As such, for the purposes of this
manuscript, we refer to this family of regular nilpotent Hessenberg
varieties as \textbf{$\mathbf{334}$-type Hessenberg varieties}.

Our main result is the following theorem. 

\begin{theorem}\label{theorem:perfect}
Let $n \geq 4$ and let $\Hess(h)$ be the $334$-type Hessenberg variety in
$\mathcal{F}\ell ags(\C^n)$.
Let $\roll: \Hess(h)^{S^1} \to S_n$ be the dimension pair
algorithm defined in Section~\ref{sec:aba}. 
Then 
\begin{equation}\label{eq:upper}
p_{\roll(w)}(w) \neq 0, \quad \textup{ and } \quad p_{\roll(w)}(w') = 0
\textup{ if } w \not \leq w' 
\end{equation}
for all $w, w'  \in \Hess(h)^{S^1}$. 
In particular 
the Hessenberg Schubert classes $\{p_{\roll(w)} \hsm
\vert \hsm w \in \Hess(h)^{S^1}\}$ form a $H^*_{S^1}(\pt)$-module
basis for the $S^1$-equivariant cohomology ring
$H^*_{S^1}(\Hess(h))$. 
\end{theorem}

For ease of exposition, and because the arguments required are of a
somewhat different nature, we prove Theorem~\ref{theorem:perfect} by
proving the two assertions in~\eqref{eq:upper} separately, as
follows.

\begin{proposition}\label{proposition:prollw at w nonzero} 
  Let $n, h, \Hess(h)$ and $\roll$ be as above. Then 
  \begin{equation}
    \label{eq:prollw at w nonzero}
    p_{\roll(w)}(w) \neq 0 
  \end{equation}
for all $w \in \Hess(h)^{S^1}$. 
\end{proposition}

\begin{proposition}\label{proposition:upper vanishing}
  Let $n, h, \Hess(h)$ and $\roll$ be as above. Then 
  \begin{equation}
    \label{eq:upper vanishing}
    p_{\roll(w)}(w') = 0
\textup{ if } w \not \leq w' 
  \end{equation}
for all $w, w' \in \Hess(h)^{S^1}$. 
\end{proposition}

The proof of Proposition~\ref{proposition:prollw at w nonzero} is the
content of Section~\ref{sec:combinatorics}. The main result of the
present section is the upper-triangularity property asserted in
Proposition~\ref{proposition:upper vanishing}. 
Its proof requires 
a number of preliminary
results. We first begin by reformulating the problem
in terms of Bruhat relations among the fixed points.

\begin{lemma}\label{lemma:fixed points versus fillings} 
Let $n, h, \Hess(h)$ and $\roll$ be as above. If for all $w, w' \in
\Hess(h)^{S^1}$ we have 
\begin{equation}\label{eq:fillings Bruhat order}
\roll(w) \leq w' \Leftrightarrow w \leq w'
\end{equation}
in Bruhat order,
then the Hessenberg
Schubert classes $\{p_{\roll(w)} \hsm
\vert \hsm w \in \Hess(h)^{S^1}\}$ satisfy~\eqref{eq:upper vanishing}. 
\end{lemma}

\begin{proof}
Recall that the equivariant Schubert classes are
poset-upper-triangular with respect to Bruhat order on $S_n$. In
particular, for all $w, w' \in S_n$ we have $\sigma_w(w') = 0$ if $w' \not \geq w$. Since the Hessenberg
Schubert classes are images of the Schubert classes and the diagram 
\begin{equation}\label{eq:comm diagram for Hess}
 \xymatrix{
H^*_{T^n}(\Flags(\C^n)) \ar @{^{(}->}[r] \ar[d] & H^*_{T^n}((\Flags(\C^n))^{T^n}) \cong
\bigoplus_{w \in W} H^*_{T^n}(\pt) \ar[d] \\
H^*_{S^1}(\Hess(h)) \ar @{^{(}->}[r] &
H^*_{S^1}((\Hess(h))^{S^1}) \cong \bigoplus_{w \in \Hess(h)^{S^1}}
H^*_{S^1}(\pt) 
}
\end{equation}
commutes, it follows that if for all $w, w' \in \Hess(h)^{S^1}$, we have 
\begin{equation}\label{eq:inequalities with fixed points} 
\roll(w) \leq w' \Leftrightarrow w \leq w'
\end{equation}
in Bruhat order then~\eqref{eq:upper vanishing} follows.  
\end{proof}

The rest of this section is devoted to the proof of~\eqref{eq:fillings
  Bruhat order}, which by Lemma~\ref{lemma:fixed points versus
  fillings} then proves Proposition~\ref{proposition:upper
  vanishing}.

\subsection{Fixed points and associated subsets for the
  $334$-type Hessenberg variety}

In this section we enumerate the fixed points in
the $334$-type Hessenberg variety and also associate to each fixed
point in 
$\Hess(h)^{S^1}$ a subset of $\{1,2,\ldots,n-1\}$. As we show
below, the set of fixed points in the Peterson variety is a subset of
the fixed points of the $334$-type Hessenberg variety, so the main
task is to describe the new fixed points which arise in the $334$-type
case. We begin with a general observation.

\begin{lemma}\label{lemma:inclusion of Hess} 
Let $n \in \N$ and let $h, h': \{1,2,\ldots,n\} \to \{1,2,\ldots, n\}$ be
two Hessenberg functions. If $h(i) \geq h'(i)$ for all $i, 1 \leq i
\leq n$, then 
\[
\Hess(h') \subseteq \Hess(h). 
\]
The inclusion $\Hess(h') \into \Hess(h)$ is $S^1$-equivariant and 
in particular $\Hess(h')^{S^1} \subseteq \Hess(h)^{S^1}$ and 
$\PFill(h') \subseteq \PFill(h)$. 
\end{lemma}

\begin{proof}
  Let $V_{\bullet} = (V_i)$ denote an element in $\mathcal{F}\ell
  ags(\C^n)$. By definition the regular nilpotent Hessenberg
  variety $\Hess(h')$ associated to $h'$ is
\begin{equation}
\Hess(h')  := \{ V_{\cdot} \in \mathcal{F}\ell ags(\C^n) \hsm \vert \hsm
NV_i \subseteq V_{h'(i)}, \textup{ for all } 1 \leq i \leq n \}
\end{equation}
where $N$ is the principal nilpotent operator. 
Since $V_i \subseteq V_{i+1}$ for all $1 \leq i \leq n-1$ by
definition of flags and $V_n = \C^n$ for all flags, if $h'(i) \leq
h(i)$ for all $i$ then $NV_i \subseteq V_{h'(i)}$ automatically
implies $NV_i \subseteq V_{h(i)}$. We conclude $\Hess(h') \subseteq
\Hess(h)$. The 
$S^1$-equivariance of 
the inclusion $\Hess(h') \into \Hess(h)$ follows from the definition
of the 
$S^1$-action of~\eqref{eq:def-circle}.
\end{proof}

Applying Lemma~\ref{lemma:inclusion of Hess} to the Hessenberg
function 
\begin{equation}\label{def: peterson hessenberg}
h'(i) = i+1 \textup{ for } 1 \leq i \leq n-1, h'(n)=n
\end{equation}
corresponding to the Peterson variety $\Hess(h')$ 
and $h$ the $334$-type Hessenberg
function~\eqref{eq:334 Hessenberg}, we conclude that all fixed points in
$\Hess(h')^{S^1}$ also arise as fixed points in
$\Hess(h)^{S^1}$. We refer to the elements of $\Hess(h')^{S^1}$
(viewed as elements of $\Hess(h)^{S^1}$) as \textbf{Peterson-type
fixed points}.
It therefore remains to describe $\Hess(h)^{S^1} \setminus
\Hess(h')^{S^1}$. It turns out to be convenient to do this by first
describing $\PFill(h) \setminus \PFill(h')$.

We first introduce some terminology. Given a permutation \(w = (w(1)
\hsm w(2) \hsm \cdots w(n))\) in one-line notation and some $i, \ell$,
we say that the entries $\{w(i), w(i+1), \ldots, w(i+\ell)\}$ form a
\textbf{decreasing staircase}, or simply a \textbf{staircase}, if
$w(j+1) = w(j)-1$ for all $i \leq j < i + \ell$. For example for $w =
4327516$, the segment $432$ is a staircase, but $751$, though the entries
decrease, is not. We will say that a consecutive series of staircases
is an \textbf{increasing sequence of staircases} (or simply
\textbf{increasing staircases}) if each entry in a given staircase is
smaller than any entry in any following staircase (reading from left
to right). For instance, \(w = 654987321\) is a sequence of staircases
$654$, $987$, and $321$, but is not an increasing sequence of
staircases since the entries $4,5,6$ are not smaller than the entries
in the later staircase $321$. However, $w = 321654987$ is an
increasing sequence of (three) staircases $321$, $654$, and $987$.

It is shown in \cite{HarTym09} that the $S^1$-fixed points of the
Peterson variety $\Hess(h')$ consist precisely of those permutations
\(w \in S_n\) such that the one-line notation of $w$ is an increasing
sequence of staircases. Since such $w$ are equal to their own
inverses, the permissible fillings $\PFill(h')$ corresponding to
$\Hess(h')$ are precisely those which are increasing sequences of
staircases (cf. Remark~\ref{remark:PF_w in case N principal}).  We now
describe the permissible fillings $\PFill(h)$ which are \emph{not}
Peterson-type fillings. We use the language of $h$-tableau trees
introduced by Mbirika; see \cite[Section 3.1]{Mbirika:2010}
for definitions.  Recall from Remark~\ref{remark:regular case} that we
identify permissible fillings with permutations in $S_n$ via one-line
notation.

\begin{lemma}\label{lemma:non-Peterson fixed points} 
Let $n \geq 4$ and let $\Hess(h)$ be the $334$-type Hessenberg variety in
$\mathcal{F}\ell ags(\C^n)$. Let $w \in \PFill(h)$ be a permissible
filling for $\Hess(h)$ which is not of Peterson type, i.e., \(w \in
\PFill(h) \setminus \PFill(h').\) 
Then precisely one of the following hold: 
\begin{itemize}
  \item The one-line notation of $w$ is of the form 
\[
w' \hsm 3 \hsm 1 \hsm 2 \hsm  w'' 
\]
where $w'$ is a (possibly empty) staircase such that $w' \hsm 3$
is also a staircase, and $w''$ is an increasing sequence of staircases. We
refer to these as \textbf{$312$-type permissible fillings.}

 \item The one-line notation of $w$ is of the form 
\[
2 \hsm w' \hsm 3 \hsm 1 \hsm w''
\]
where $w'$ is a (possibly empty) staircase such that $w' \hsm 3$
is also a staircase, and $w''$ is an increasing sequence of staircases. We
refer to these as \textbf{$231$-type permissible fillings.} 
\end{itemize} 
Moreover, any filling satisfying either of the above conditions appears 
in $\PFill(h) \setminus \PFill(h')$.
\end{lemma}

\begin{proof}[Proof of Lemma~\ref{lemma:non-Peterson fixed points}]
For any Hessenberg function \(h: \{1, 2, \ldots, n\} \to \{1,2,\ldots,
n\},\) Mbirika shows in \cite[Section 3.2]{Mbirika:2010} that the Level $n$ fillings
in an $h$-tableau tree are precisely the permissible fillings with
respect to $h$.
For the Peterson Hessenberg function in~\eqref{def: peterson hessenberg}
Mbirika's corresponding $h$-tableau tree has the property
that for every $k$ with $1 \leq k \leq n-1$ and every vertex at
Level $k$, there are precisely $2$ edges going down from that vertex
to a Level $k+1$ vertex. (This is because the corresponding
\emph{degree tuple} $\beta$ \cite[Definition 3.1.1]{Mbirika:2010} has
$\beta_i=2$ for all $1 \leq i \leq n-1$.) In the case of the $334$-type Hessenberg
function, by definition the $h$-tableau tree also has precisely $2$ edges going down
from every vertex at Level $k$ for all $k \neq 2$, $1 \leq k \leq
n-1$. However, at Level 2, each vertex has not $2$ but $3$ edges pointing down
to a vertex at Level 3. 

From \cite[Section 3]{Mbirika:2010} (cf. in particular \cite[Definition
3.1.9]{Mbirika:2010}) it can be seen that for the case of the Peterson
Hessenberg function, the corresponding $h$-tableau tree at Level 2 has
vertices $\bullet \, 2 \, 1 \, \bullet$ and $1 \bullet 2 \, \bullet$, whereas for
the $334$-type Hessenberg function, the Level 2 vertices have the form
$\bullet \, 2 \bullet 1 \, \bullet$ and $\bullet \, 1 \bullet 2 \, \bullet$. Here
the bullets indicate the locations of the $h$-permissible positions
available for the placement of the next index $3$, in the sense of
\cite[Section 3]{Mbirika:2010} (cf. in particular \cite[Lemma
3.1.8]{Mbirika:2010}). In particular, since we saw above that the edges going
down from Level 3 onwards are identical in both the Peterson and
$334$-type Hessenberg case, it follows that the branches of the tree
emanating downwards from the two Level 3 vertices $3 \, 2 \, 1 \,
\bullet$, $2 \, 
1 \, 3 \, \bullet$ 
(coming from $\bullet \, 2 \bullet 1 \, \bullet$) and the two vertices
$1 \bullet 3 \, 2 \, \bullet$, $1 \, 2 \bullet 3 \, \bullet$ (coming
from $\bullet \, 1 \bullet 2 \, \bullet$) are identical to
the corresponding branches in the $h$-tableau tree for the Peterson
Hessenberg function. Hence all
permissible fillings at the final Level $n$ of these branches are
of Peterson type. In contrast, the branches emanating from $2 \bullet
3 \, 1 \, \bullet$
and $\bullet \, 3 \, 1 \, 2 \, \bullet$ do not appear in the Peterson
$h$-tableau tree, and none of the fillings appearing at Level
$n$ in these branches can be Peterson permissible fillings since a $3$ appears directly
before a $1$. Hence it is precisely these branches which account for
the permissible fillings which are not of Peterson type. As noted above, the rest
of the branch only has 2 edges going down from each vertex with
$h$-permissible positions determined exactly as in the Peterson
case. In particular, except for the exceptional $3$ appearing directly
to the left of a $1$, the fillings must consist of decreasing
staircases and all possible arrangements of decreasing staircases do appear. The result follows.

\end{proof}

\begin{example}
  Suppose $n=8$. Then $w = 54312876$ is an example of a $312$-type
  permissible filling where $w'=54$ and $w'' = 876$. An example of a
  $231$-type permissible filling is $w = 25431876$ where $w'=54$ and
  $w''=876$. Neither of these are permissible with respect to the
  Peterson Hessenberg function $h'$ since a $3$ appears directly to
  the left of a $1$. Nevertheless, both of these fillings are closely
  related to the Peterson-type permissible filling $w = 54321876$;
  this relationship is closely analyzed and used below.
\end{example}

We now give explicit descriptions of the corresponding
non-Peterson-type elements in
$\Hess(h)^{S^1}$, obtained by taking inverses of the permissible
fillings described in Lemma~\ref{lemma:non-Peterson fixed points}. 

\begin{definition}\label{definition:nonPeterson fixed points}
  Let $w \in \Hess(h)^{S^1}$. We say $w$ is a \textbf{$312$-type (respectively
  $231$-type) fixed point}
  if its inverse $w^{-1}$ is a permissible filling of 
  $312$-type (respectively $231$-type). 
\end{definition}

As observed above, 
since Peterson-type permissible fillings are equal to their own
inverses, in that case there is no distinction between the fillings
and their associated fixed points. For the $312$ and $231$-types,
however, this is not the case. We record the following. The proof is a 
straightforward computation and is left to the reader.

\begin{lemma}\label{lemma:one line notation for nonPeterson}
Let $w$ be a $312$-type (respectively $231$-type) 
permissible filling. Let $a_2$ be the integer such that $a_2+1$ is the
first entry (respectively second entry) in the one-line notation of
$w$. Let $w^{-1}$ be the corresponding $312$-type (respectively $231$
type) fixed point. Then: 
\begin{itemize}
\item the one-line
  notation of $w^{-1}$ is the same as that of $w$ for all $\ell$-th entries
  with $\ell > a_2+1$, 
\item if $w$ is $312$-type,  then the first $a_2+1$ entries of the
  one-line notation of $w^{-1}$ are
\begin{equation}\label{eq:312 one line notation}
a_2 \hsm a_2+1 \hsm a_2-1 \hsm a_2 - 2 \hsm \cdots \hsm 2 \hsm 1
\end{equation}
\item if $w$ is $231$-type, then the first $a_2+1$ entries of the
  one-line notation of $w^{-1}$ are
\begin{equation}\label{eq:231 one line notation}
a_2+1 \hsm 1 \hsm a_2 \hsm a_2 - 1 \hsm \cdots 3 \hsm 2 
\end{equation}
\end{itemize}
\end{lemma}

In the case of the Peterson variety, there is a convenient bijective correspondence
between the set of $S^1$-fixed points of the Peterson variety 
and subsets $\mathcal{A}$ of $\{1,2,\ldots,n-1\}$ given as follows
\cite[Section 2.3]{HarTym09}.
Let $w$ be a Peterson-type fixed point.
Then the corresponding subset is 
\begin{equation}\label{eq:Peterson subset} 
\mathcal{A} := \{i : 1 \leq i \leq n-1 \textup{ and } w(i) = w(i + 1) + 1 \}
\subseteq⊆ \{1, 2, \ldots , n-1\}. 
\end{equation}
In the case of the $334$-type Hessenber variety, it is also useful to
assign a subset of $\{1,2,\ldots,n-1\}$ to each fixed point as
follows.

\begin{definition}\label{definition: associated subset}
Let $w \in \Hess(h)^{S^1}$. The \textbf{associated subset of
  $\{1,2,\ldots,n\}$ corresponding to
  $w$}, notated $\mathcal{A}(w)$, is defined as follows: 
\begin{itemize}
\item Suppose $w$ is of Peterson type. Then $\mathcal{A}(w)$
  is defined to be the set $\mathcal{A}$ in~\eqref{eq:Peterson subset}. 
\item Suppose $w$ is $312$-type. Consider the permutation  
  $w' := w s_1$ (i.e. swap the $a_2$ and the $a_2+1$ in the one-line
  notation~\eqref{eq:312 one line notation}). This is a fixed point of Peterson type. Define
  $\mathcal{A}(w) := \mathcal{A}(w')$.
 \item Suppose $w$ is $231$-type. Consider the permutation 
\[
w' = w s_2 s_3 \cdots s_{a_2}
\]
(i.e. move the $1$ to the right of the $2$ in the one-line
notation~\eqref{eq:231 one line notation}). This is a fixed point of Peterson type. Define
  $\mathcal{A}(w) := \mathcal{A}(w')$. 
\end{itemize}
\end{definition}

\begin{example}\label{example:Aw}
Suppose $n=8$. 
\begin{itemize}
\item Suppose $w$ is the Peterson-type fixed point $w =
  54321876$. Then $\mathcal{A}(w) = \{1,2,3,4\} \cup \{6,7\}$. This
  agrees with the association $w \mapsto \mathcal{A}(w)$ used in
  \cite{HarTym09}. 
\item Suppose $w$ is the $312$-type fixed point $w = 34217658$
  (corresponding to the $312$-type permissible filling
  $43127658$). Then $w' = w s_1 = 43217658$ and $\mathcal{A}(w) :=
  \mathcal{A}(w') = \{1,2,3\} \cup \{5,6\}$.
\item Suppose $w$ is the $231$-type fixed point $w =
  51432768$ (corresponding to the $231$-type permissible filling
  $25431768$). 
Then $w' = 54321768$ and \(\mathcal{A}(w) :=
  \mathcal{A}(w') = \{1,2,3,4\}\cup \{6\}.\) 
\end{itemize}
\end{example} 

\begin{remark}\label{remark: not injective} 
The three fixed points $w =
54321876$, $w = 45321876$, and $w = 51432876$, which are respectively of Peterson type, $312$
type, and $231$-type, all have the same
associated subset $\mathcal{A}(w) = \{1,2,3,4\} \cup
\{6,7\}$. 
\end{remark}

It is useful to observe that the $312$-type and $231$-type fixed points
have associated subsets that always contain $1$ and $2$. 

\begin{lemma}\label{lemma:contains 1 and 2}
Let $w$ be a $334$-type Hessenberg fixed point. Suppose
further that $w$ is not of Peterson type. Then $\{1,2\} \subseteq \mathcal{A}(w)$.
\end{lemma}

\begin{proof}
  From the explicit descriptions of the one-line notation of the $312$
  type (respectively $231$-type) fixed points given above, we know that the
  initial segment $a_2 \, a_2+1 \, \cdots \, 2 \, 1$ (respectively
  $a_2+1 \, 1 \, a_2 \, \cdots \, 3 \, 2$) in the one-line notation is such that $a_2 \geq
  2$. From Definition~\ref{definition: associated subset} it follows
  that the first decreasing staircase of the associated Peterson-type
  fixed point $ws_1$ (respectively $w s_2 s_3 \cdots s_{a_2}$) is of
  length at least $3$.  In particular, the first staircase starts with
  an integer $k$ which is $\geq 3$. The result follows.
\end{proof}

As noted in Remark~\ref{remark: not injective},
the association $w \mapsto \mathcal{A}(w)$ given in
Definition~\ref{definition: associated subset} is \emph{not} 
one-to-one and hence in particular not a bijective correspondence. This makes our analysis more
complicated than in \cite{HarTym09}, but the notion is still useful for our arguments below.

\subsection{Reduced word decompositions for $334$-type fixed points
  and rolldowns}

In this section we fix particular choices of reduced word
decompositions for the fixed points in $\Hess(h)^{S^1}$ which we use
in our arguments below. We also compute, and fix choices of
reduced words for, the rolldowns $\roll(w)$ of the fixed points. 

The association $w \mapsto \mathcal{A}(w)$ of the previous section
allows us to describe these reduced word decompositions in relation to
that of the Peterson-type fixed points.  Let $a$ be a positive integer
and $k$ a non-negative integer. Recall that a reduced word
decomposition of the maximal element (the full inversion) 
in the subgroup $S_{\{a,a+1,\ldots,a+k+1\}}
\subseteq S_n$
is given by 
\begin{equation}\label{eq:standard reduced word}
s_a (s_{a+1} s_a) (s_{a+2} s_{a+1} s_a) \cdots (s_{a+k} s_{a+k-1} \cdots s_{a+1} s_a).
\end{equation}
For the purposes of this manuscript, we call this the \emph{standard
  reduced word (decomposition)} for the maximal element. (This is 
different from the choice of reduced word decomposition used in
\cite[Section 2.3]{HarTym09}.) 
We denote a consecutive set of integers $\{a,
a+1,\ldots, a+k\}$ for $a$ positive and $k$ a non-negative
integer by $[a, a+k]$. We say that $[a,a+k]$ is a \textbf{maximal consecutive
substring} of $\mathcal{A}$ if $[a,a+k] \subseteq \mathcal{A}$ and
neither $a-1$ nor $a+k+1$ are in $\mathcal{A}$. It is straightforward
that any subset $\mathcal{A}$ of $\{1,2,\ldots,
n-1\}$ uniquely decomposes into a disjoint union of maximal
  consecutive substrings
\begin{equation}\label{eq:A into substrings}
{\mathcal{A}} = [a_1, a_2] \cup [a_3,a_4] \cup \cdots \cup
[a_{m-1},a_m].
\end{equation}
For instance, for $\mathcal{A} = \{1,2,3, 5,6,9,10,11\}$, the
decomposition is $\mathcal{A} = [1,3] \cup [5,6] \cup [9,11]$. For any
$[a,b]$, denote by $w_{[a,b]}$ the full inversion in the subgroup
$S_{[a,b+1]}$. Then it follows from Definition~\ref{definition:
  associated subset} (see also \cite[Section 2.3]{HarTym09}) that 
the Peterson-type fixed point associated to
$\mathcal{A}$, which we denote by $w_{\mathcal{A}}$, is the product
\begin{equation}\label{eq:wA-reduced-word} 
  w_{\mathcal{A}} := w_{[a_1, a_2]} w_{[a_3, a_4]} w_{[a_5, a_6]}
  \cdots w_{[a_{m-1}, a_m]}. 
\end{equation}
We fix a choice of reduced word
decomposition of $w_{\mathcal{A}}$ given by taking the product of the
standard reduced words~\eqref{eq:standard reduced word} for each of the full inversions $w_{[a_j,
  a_{j+1}]}$ appearing in~\eqref{eq:wA-reduced-word}. For the purposes
of this manuscript we call this the
\textbf{standard reduced word decomposition of a Peterson-type
fixed point $w_{\mathcal{A}}$}.

\begin{example}\label{example:Peterson reduced word}
  Let $n=7$ and let $w = 4321765$ be a Peterson-type fixed point.
  Then the two decreasing staircases are $4321$ and $765$, the
  associated subset $\mathcal{A}(w)$ is $\{1,2,3\} \cup \{5,6\}$ with
  maximal consecutive strings $[1,3]:=\{1,2,3\}$ and
  $[5,6]:=\{5,6\}$. The standard reduced word decomposition of $w$ is
\begin{equation}\label{eq:wprime reduced word}
w_{\{1,2,3\}\cup\{5,6\}} = w_{[1,3]}w_{[5,6]} = s_1 (s_2 s_1) (s_3 s_2 s_1) s_5 (s_6 s_5).
\end{equation}
\end{example}

We now fix a reduced word decomposition of the non-Peterson-type
fixed points.

\begin{lemma}\label{lemma:nonPeterson reduced words}
Let $w \in \Hess(h)^{S^1}$ be a fixed point which is not of Peterson
type and let $\mathcal{A}(w) = [a_1, a_2] \cup [a_3, a_4] \cup \cdots
\cup [a_{m-1}, a_m]$ be the associated subset with its decomposition
into maximal consecutive substrings.
\begin{itemize} 
\item If $w$ is $312$-type then a reduced word decomposition for $w$ is given by 
\begin{equation}\label{eq:312 fixed point reduced word}
s_1 (s_2 s_1) \cdots (s_{a_2} s_{a_2-1} \cdots s_3 s_2) w_{[a_3,a_4]}
\cdots w_{[a_{m-1},a_m]}
\end{equation}
and 
\item if $w$ is $231$-type then a reduced word
  decomposition for $w$ is given by
\begin{equation}\label{eq:231 fixed point reduced word}
   s_2 (s_3 s_2) \cdots (s_{a_2-1} s_{a_2-2} \cdots s_3 s_2) (s_{a_2}
   s_{a_2-1} \cdots s_2 \hsm s_1) w_{[a_3,a_4]}
\cdots w_{[a_{m-1},a_m]} 
\end{equation}
\end{itemize} 
where the $w_{[a_\ell, a_{\ell+1}]}$ in the above expressions are assumed to be given the reduced word
decomposition described in~\eqref{eq:standard reduced word}.
\end{lemma}

\begin{proof}
For the first assertion, observe that the explicit description of
the one-line notation $312$-type fixed points in~\eqref{eq:312
  one line notation} implies that $w$ has precisely $1$ fewer
    inversion than $w_{\mathcal{A}(w)}$. An explicit computation shows
    that the given word~\eqref{eq:312 fixed point reduced word} is equal to $w$, so it is a word
    decomposition of $w$ with exactly as many simple transpositions as
    the Bruhat length of $w$. In particular it must be reduced. A
    similar argument proves the second assertion. 
\end{proof}

\begin{example}\label{example:nonPeterson reduced word}
Suppose $n=7$. Suppose $w = 3421765$ is a $312$-type fixed point.
Then the reduced word decomposition of $w$ given in
Lemma~\ref{lemma:nonPeterson reduced words} is 
\[
w = s_1 (s_2 s_1) (s_3 s_2) s_5 (s_6 s_5). 
\]
Similarly
suppose $w = 4132765$ is a $231$-type fixed point. Then the
reduced word decomposition of $w$ given in Lemma~\ref{lemma:nonPeterson
  reduced words} is 
\[
w = s_2 (s_3 s_2 s_1) s_5 (s_6 s_5). 
\]
\end{example}

Henceforth we always use the reduced words given above.

Next we explicitly describe the rolldowns
$\roll(w)$ associated
to each $w$ in $\Hess(h)^{S^1}$ by the dimension pair algorithm. 
We begin with the Peterson-type
fixed points. It turns out there are two important subcases of
Peterson-type fixed points. 

\begin{definition}\label{def:contains or not}
We say that a Peterson-type fixed point $w$
\textbf{contains the string $321$} (or simply \textbf{contains $321$}) if, in the one-line notation of $w$, the
string $321$ appears (equivalently, if $\{1,2\} \subseteq
\mathcal{A}(w)$). We say $w$ \textbf{does not contain the string $321$} (or simply \textbf{does not contain $321$})
otherwise. 
\end{definition} 

\begin{remark}
Note that Definition~\ref{def:contains or not} is different from the standard notion of pattern-containing or pattern-avoiding permutations since here we require the one-line notation of $w$ to contain the string $321$ exactly.
\end{remark}

Given a subset $\mathcal{A} = \{j_1 < j_2 < \cdots <
j_k\} \subseteq \{1,2,\ldots, n-1\}$ and corresponding Peterson-type
fixed point $w_{\mathcal{A}}$, we call the 
permutation
\begin{equation}\label{eq:Peterson case rolldown}
s_{j_k} s_{j_{k-1}} \cdots s_{j_2} s_{j_1} \in S_n
\end{equation}
the \textbf{Peterson case rolldown of $w_{\mathcal{A}}$.} Note that the
word~\eqref{eq:Peterson case rolldown} is in fact a reduced word
decomposition of this permutation; we always use this choice of
reduced word. The terminology is motivated by the fact
that~\eqref{eq:Peterson case rolldown} is the (inverse of the)
permutation given in \cite[Definition 4.1]{HarTym09}.
(The fact that it is the inverse of 
the permutation used in \cite{HarTym09} does not affect the theory very much, as is
explained in \cite[Proposition 5.16]{HarTym09}.)

\begin{lemma}\label{lemma:rolldowns for Petersons}
Let $n \geq 4$ and $\Hess(h)$ the $334$-type Hessenberg variety in $\Flags(\C^n)$. 
Let $w$ be a Peterson-type
fixed point and let $\mathcal{A}(w) = \{j_1 < j_2 < \cdots < j_k\}$ be its
associated subset. 
\begin{itemize}
\item  Suppose $w$ does not contain $321$.
Then 
 $\roll(w)$ 
is the Peterson case rolldown of
  $w_{\mathcal{A}(w)}$. 
\item Suppose $w$ does contain $321$, i.e., $\mathcal{A}(w) = \{j_1 <
  j_2 < \cdots < j_k\}$ for $k \geq 2$ and $j_1 = 1$ and $j_2 =
  2$. Then $\roll(w)$ 
is
\begin{equation}\label{eq:321 rolldown} 
\roll(w) = s_{j_k} s_{j_{k-1}} \cdots s_{j_3} s_1 s_2 s_1. 
\end{equation}
\end{itemize}
In particular, if a Peterson-type fixed point $w$ contains $321$,  
then its rolldown $\roll(w)$ is Bruhat-greater, and has Bruhat length $1$
greater, than the Peterson case
rolldown of $w$. 
\end{lemma}

\begin{proof} 
If $w$ contains a $321$, then by
Definition~\ref{definition:dimension pair}, the pairs $(1,3), (2,3)$ and
$(1,2)$ are all dimension pairs in $w$. Hence $3$ appears precisely twice as a top part of
a dimension pair and $2$ appears precisely once. Thus 
by construction the dimension pair algorithm the permutation
$\omega(\mathbf{x})$ begins with the word $s_1 (s_2 s_1)$. With respect to all other indices $j
\in \mathcal{A}(w)$, the $334$-type Hessenberg function is identical
to the Peterson Hessenberg function and hence for each such $j$, the
index $j+1$ appears
precisely once as a top part of a dimension pair of $w$ and thus
contributes precisely one $s_j$ to $\omega(\mathbf{x})$. 
Taking the inverse 
yields~\eqref{eq:321 rolldown} as desired. 

If $w$ does not contain $321$, then $3$ appears
at most once as the top part of a dimension pair in $w$, and again for
all other indices the computations are identical to the Peterson case
as above. Hence $\roll(w)$ is identical to the
Peterson case rolldown. This completes the proof. 
\end{proof}

Next, we give an explicit description, along with a choice of reduced word
decomposition, of the rolldowns corresponding to the non-Peterson-type
fixed points.

\begin{lemma}\label{lemma:rolldowns for nonPetersons}
Let $w \in \Hess(h)^{S^1}$ and suppose
that $w$ is not of Peterson type. Let $\mathcal{A}(w) = \{j_1 =1 <
j_2=2 < j_3 < \cdots < j_k\}$ for some $k \geq 2$. 
\begin{enumerate}
\item If $w$ is of 
$312$-type, then 
 the dimension pair algorithm 
  associates to $w$ the permutation
\begin{equation}\label{eq:rolldown for 312}
\roll(w) = s_{j_k} s_{j_{k-1}} \cdots s_{j_4} s_{j_3} s_1 s_2. 
\end{equation}
\item If $w$ is $231$-type, then  the dimension pair algorithm 
  associates to $w$ the permutation
\begin{equation}\label{eq:rolldown for 231}
\roll(w) = s_{j_k} s_{j_{k-1}} \cdots s_{j_4} s_{j_3} s_2 s_1. 
\end{equation}
\end{enumerate}
\end{lemma}

\begin{proof}
  Suppose $w$ is a $312$-type fixed point so
  $\phi_\lambda^{-1}(w^{-1})$ is a $312$-type permissible filling. By
  definition of dimension pairs, $2$ does not appear as the top part
  of any dimension pair (since it appears to the right of a $1$). Also
  by definition, $3$ appears as a top part of the two dimension pairs
  $(1,3)$ and $(2,3)$. The form of the $312$-type permissible fillings
  described in Lemma~\ref{lemma:non-Peterson fixed points} and the
  definition of $\mathcal{A}(w)$ imply that the other dimension pairs
  are precisely the pairs $(j, j+1)$ for $j \in \mathcal{A}(w)$ (for
  $j \neq 1,2$), from which it follows that $\omega(\mathbf{x}) = s_2
  s_1 s_{j_3} s_{j_4} \cdots s_{j_{k-1}} s_{j_k}$. Taking inverses
  yields~\eqref{eq:rolldown for 312}. The proof of the second
  assertion is similar. 
\end{proof}

\begin{example}
  \begin{itemize}
  \item Suppose $w =  54321876$. This is of Peterson type. Then 
   $\roll(w) = (\omega(\mathbf{x}))^{-1} = s_7 s_6 s_4 s_3 s_1 s_2 s_1$. 
  \item Suppose $w=45321876$. This is $312$-type.
    Then $\roll(w) = (\omega(\mathbf{x}))^{-1} = s_7 s_6 s_4 s_3 s_1 s_2$. 
 \item Suppose $w = 51432876$. This is $231$-type. 
   Then $\roll(w) = (\omega(\mathbf{x}))^{-1} = s_7 s_6 s_4 s_3 s_2 s_1$. 
 \end{itemize}
\end{example}

We conclude the section with a computation of the one-line notation of
the rolldowns for different types; we leave proofs to the reader.

\begin{lemma}\label{lemma:one line notation for rolldowns}
  Let $w$ be a $334$-type Hessenberg fixed point and let $\mathcal{A}(w) = [a_1, a_2] \cup
\cdots \cup [a_{m-1}, a_m]$ be its associated subset with its
decomposition into maximal consecutive substrings. Suppose $w$ is of
Peterson type that contains $321$, $312$-type, or of $231$
type. Then $a_1 = 1$, $a_2 \geq 2$ and the first $a_2+1$ entires of
the one-line notation of $\roll(w)$ is 
\begin{equation}\label{eq:rolldown one line notation for Peterson 321}
a_2+1 \, 2 \, 1 \, 3 \, 4 \, \cdots \, a_2 
\end{equation}
for $w$ of Peterson type that contains $321$, 
\begin{equation}
  \label{eq:rolldown one line notation for 312}
  2 \, a_2 +1 \, 1 \, 3 \, 4 \, \cdots \, a_2 
\end{equation}
for $w$ $312$-type, and 
\begin{equation}
  \label{eq:rolldown one line notation for 231}
  a_2+1 \, 1 \, 2\, 3 \, \cdots \, a_2
\end{equation}
for $w$ $231$-type. 
\end{lemma}

\subsection{Bruhat order relations} 

In this section we analyze the properties of the association 
\(w \mapsto \mathcal{A}(w)\) with respect to comparisons in Bruhat
order.

The first two lemmas are straightforward and proofs left to the reader. 

\begin{lemma}\label{lemma:wA maximal}
Let $\mathcal{A} \subseteq \{1,2,\ldots, n-1\}$ and let
$w_{\mathcal{A}}$ be the Peterson-type filling associated to
$\mathcal{A}$. Then $w_{\mathcal{A}}$ is maximal in the subgroup
$S_{\mathcal{A}}$ of
$S_n$ generated by the simple transpositions $\{s_i\}_{i \in
  \mathcal{A}}$. In particular, $w_{\mathcal{A}}$ is Bruhat-bigger
than any permutation $w \in S_{\mathcal{A}}$. 
\end{lemma}

\begin{lemma}\label{lemma:nonPetersons less than associated Petersons}
  Let $w \in \Hess(h)^{S^1}$. Suppose $w$ is not of Peterson
  type. Then $w$ is Bruhat-less than the Peterson type fixed point
  $w_{\mathcal{A}(w)}$ 
  corresponding to $\mathcal{A}(w)$.
\end{lemma}

We also observe that a Bruhat relation $w<w'$ implies a containment
relation of the associated subsets. 

\begin{lemma}\label{lemma:Aw and transpositions}
  Let $w, w' \in \Hess(h)^{S^1}$ and let 
  $\mathcal{A}(w), \mathcal{A}(w')$
  be the respective associated subsets. Let $s_i$ be a simple transposition. Then: 
\begin{enumerate} 
\item $s_i < w$ if and only if $i \in \mathcal{A}(w)$, 
\item $s_i < \roll(w)$ if and only if $i \in \mathcal{A}(w)$, 
\item if $w \leq w'$ or $\roll(w) \leq w'$ then $\mathcal{A}(w) \subseteq \mathcal{A}(w')$.
\end{enumerate}
\end{lemma}

\begin{proof}
  Bruhat order is independent of choice of reduced word decomposition
  for $w$. Therefore a simple transposition $s_i$ is less than $w$ in
  Bruhat order if and only if $s_i$ appears in a (and hence any)
  reduced word decomposition of $w$.  In particular, to prove the
  first claim it suffices to observe that by the definitions of
  $\mathcal{A}(w)$, the index $i$ appears in $\mathcal{A}(w)$
  precisely when $s_i$ appears in the choice of reduced word for $w$
  given above. 
 A similar 
  argument using the explicit reduced words given for $\roll(w)$
  in Lemmas~\ref{lemma:rolldowns for Petersons}
  and~\ref{lemma:rolldowns for nonPetersons} proves the second
  claim. The last claim follows from the first two. 
\end{proof}

We have just seen that $w \leq w'$ implies $\mathcal{A}(w) \subseteq
\mathcal{A}(w')$.  In the case of the Peterson variety $\Hess(h')$
these Bruhat relations are precisely encoded by the partial ordering
given by containment of the $\mathcal{A}(w)$; specifically, by
Lemma~\ref{lemma:wA maximal}, $w_{\mathcal{A}} \leq w_{\mathcal{B}}$
if and only if $\mathcal{A} \subseteq \mathcal{B}$. In our $334$-type
Hessenberg case this is no longer true, although the sets
$\mathcal{A}(w)$ do still encode the Bruhat data. The precise
statements occupy the next several lemmas.

We take a moment to recall the 
\textbf{tableau criterion} for determining Bruhat order in the
Weyl group $S_n$ (see e.g. \cite{BjoBre05}) which will be useful in
the discussion below. 
For \(w \in S_n,\) denote by $D_R(w)$ the descent set of $w$, namely,
\[
D_R(w) := \{ i \hsm \vert \hsm w(i) > w(i+1), 1 \leq i \leq n-1 \}.
\]
For example, for \(w = 368475912\) the
descent set is $D_R(w) = \{3, 5, 7\}$. 

\begin{theorem} (``The tableau criterion'' \cite[Theorem 2.6.3]{BjoBre05})
  For $w,v \in S_n$, let $w_{i,k}$ be the $i$-th element in the
  increasing rearrangement of $w(1), w(2), \ldots, w(k)$, and
  similarly for $v_{i,k}$. Then $w \leq v$ in Bruhat order if and only
  if 
\begin{equation}\label{eq:tableau}
w_{i,k} \leq v_{i,k} \textup{ for all } k \in D_R(w) \textup{ and } 1
\leq i \leq k.
\end{equation}
\end{theorem}

For example, suppose \(w = 368475912\) and \(v = 694287531.\) Since
$D_R(w) = \{3,5,7\}$, we examine the three increasing rearrangements
of initial segments of $w$ and $v$ of lengths $3, 5,$ and $7$
respectively, which we may organize into Young tableaux: 
\[
\begin{array}{cc} 
w & v \\ 
{\def\lr#1{\multicolumn{1}{|@{\hspace{.6ex}}c@{\hspace{.6ex}}|}{\raisebox{-.3ex}{$#1$}}}
\raisebox{-.6ex}{$\begin{array}[b]{ccccccc}
\cline{1-1}\cline{2-2}\cline{3-3}\cline{4-4}\cline{5-5}\cline{6-6}\cline{7-7}
\lr{3}&\lr{4}&\lr{5}&\lr{6}&\lr{7}&\lr{8}&\lr{9}\\
\cline{1-1}\cline{2-2}\cline{3-3}\cline{4-4}\cline{5-5}\cline{6-6}\cline{7-7}
\lr{3}&\lr{4}&\lr{6}&\lr{7}&\lr{8}\\
\cline{1-1}\cline{2-2}\cline{3-3}\cline{4-4}\cline{5-5}
\lr{3}&\lr{6}&\lr{8}\\
\cline{1-1}\cline{2-2}\cline{3-3}
\end{array}$}
}

&

{\def\lr#1{\multicolumn{1}{|@{\hspace{.6ex}}c@{\hspace{.6ex}}|}{\raisebox{-.3ex}{$#1$}}}
\raisebox{-.6ex}{$\begin{array}[b]{ccccccc}
\cline{1-1}\cline{2-2}\cline{3-3}\cline{4-4}\cline{5-5}\cline{6-6}\cline{7-7}
\lr{2}&\lr{4}&\lr{5}&\lr{6}&\lr{7}&\lr{8}&\lr{9}\\
\cline{1-1}\cline{2-2}\cline{3-3}\cline{4-4}\cline{5-5}\cline{6-6}\cline{7-7}
\lr{2}&\lr{4}&\lr{6}&\lr{8}&\lr{9}\\
\cline{1-1}\cline{2-2}\cline{3-3}\cline{4-4}\cline{5-5}
\lr{4}&\lr{6}&\lr{9}\\
\cline{1-1}\cline{2-2}\cline{3-3}
\end{array}$}}  

\end{array} 
\]
Comparing corresponding entries, there are two violations of the
tableau condition of the proposition $(3>2)$ in the upper-left corner,
so we conclude that $w \not < v$.

Now we observe that some Bruhat relations never arise.

\begin{lemma}\label{lemma:not bruhat less}
Let $w,w' \in \Hess(h)^{S^1}$. 
Let $\mathcal{A}(w) = [a_1, a_2] \cup
\cdots \cup [a_{m-1}, a_m]$ be the associated subset of $w$ with its
decomposition into maximal consecutive substrings. 
Suppose one of the
following conditions hold:
\begin{enumerate}
\item $w'$ is of Peterson type that does not contain $321$ while $w$ is not, 
\item $w'$ is $231$-type while $w$ is either of Peterson type 
  that contains $321$ or is $312$-type.
\end{enumerate} 
Then $w \not < w'$ and $\roll(w) \not < w'$. 
\end{lemma}

\begin{proof}
  If $w'$ is of Peterson type that does not contain $321$, then $\{1,2\} \not
  \subseteq \mathcal{A}(w')$ by definition of the associated subsets.
  All other types (Peterson type that contains $321$, or $312$-type, or
  $231$-type) have associated subsets containing $\{1,2\}$ by
  Lemma~\ref{lemma:contains 1 and 2} and by definition of
  $\mathcal{A}(w)$. The claim (1) now follows from
  Lemma~\ref{lemma:Aw and transpositions}.

  Next suppose $w'$ is $231$-type and $w$ is of Peterson type
  that contains $321$. Then the first two entries of the one-line
  notation of $w$ must be both strictly greater than $1$, and $2 \in
  D_R(w)$. Similarly if $w$ is a $312$-type fixed point then $a_2 \geq
  2$. From~\eqref{eq:312 one line notation} it follows that the
  first two entries in the one-line notation of $w$ are also strictly
  greater than $1$, and $2 \in D_R(w)$. On the other hand, the
  one-line notation for a $231$-type fixed point in~\eqref{eq:231 one line
    notation} has a $1$ in the second entry. By the tableau criterion,
  if $w<w'$ then since $2 \in D_R(w)$ in both cases under
  consideration, we must have that one of the first two entries of $w$
  is equal to $1$, but we have just seen that is impossible. Hence $w
  \not < w'$.  The assertion that $\roll(w) \not < w'$ follows by a
  similar argument using~\eqref{eq:rolldown one line notation for
    Peterson 321}, ~\eqref{eq:rolldown one line notation for 312},
  and~\eqref{eq:v_[1,a]}.
\end{proof}

For the next lemma and below, we say two fixed points are \textbf{of
  the same type} if both are Peterson-type, or both are $312$-type, or
both are $231$-type.

\begin{lemma}\label{lemma:when set containment enough}
  Let $w,w' \in \Hess(h)^{S^1}$.
Suppose one of the following conditions hold: 
\begin{itemize}
\item $w$ and $w'$ are of the same type, or 
\item $w$ is of Peterson type and does not contain $321$, and $w'$ is
  either $312$-type or $231$-type, or 
\item $w$ is either $312$-type or $231$-type, and $w'$ is of
  Peterson type. 
\end{itemize}
Then 
\[
w< w' \textup{ if and only if } \mathcal{A}(w) \subseteq
\mathcal{A}(w').
\]
\end{lemma}

\begin{proof}
  Since the lemma above shows that $w < w'$ implies $\mathcal{A}(w)
  \subseteq \mathcal{A}(w')$, for all cases it suffices to show the
  reverse implication.  First suppose $w$ and $w'$ are of the same
  type and $\mathcal{A}(w) \subseteq \mathcal{A}(w')$.  An examination
  of the reduced word decompositions of the $334$-type
  fillings given in the above discussion and an argument similar to
  that in \cite{HarTym09} implies $w<w'$.  Now suppose $w$ is of
  Peterson type and does not contain $321$ and $w'$ is either of $312$
  type or $231$-type. Then since $\{1,2\} \not \subseteq
  \mathcal{A}(w)$, either $1 \not \in \mathcal{A}(w)$ or $2 \not \in
  \mathcal{A}(w)$. From the explicit reduced word decompositions of
  $312$ or $231$-type fixed points chosen above it can be seen that
  $w'$ is Bruhat-greater than both $w_{\mathcal{A}(w') \setminus
    \{1\}}$ and $w_{\mathcal{A}(w') \setminus \{2\}}$. The claim now
  follows from Lemma~\ref{lemma:wA maximal}.  Finally suppose $w$ is
  either $312$-type or $231$-type and $w'$ is of Peterson
  type. Since $\mathcal{A}(w) \subseteq \mathcal{A}(w')$ we know from
  Lemma~\ref{lemma:wA maximal} that $w_{\mathcal{A}(w)}
  < w_{\mathcal{A}(w')} = w'$. Lemma~\ref{lemma:nonPetersons less than
    associated Petersons} shows that $w< w_{\mathcal{A}(w)}$ so the
  result follows.
\end{proof}

The next step is to show that Bruhat 
relations between certain Hessenberg fixed points are connected to
lengths of initial maximal consecutive substrings in the associated
subsets. 
We need some notation. Let $\mathcal{A} \subseteq \{1,2,\ldots,n-1\}$. Recall we denote by
$w_{\mathcal{A}}$ the Peterson-type fixed point associated to
$\mathcal{A}$. For the purposes of this discussion 
we let $u_{\mathcal{A}}$ (respectively $v_{\mathcal{A}}$) denote the
  $312$-type (respectively $231$-type) fixed point with associated
subset $\mathcal{A}$. Thus for $\mathcal{A} = [1,a]$ for some $a$ with
$2 \leq a \leq n-1$, we have 
  \begin{equation}
    \label{eq:u_[1,a]}
    \begin{split}
      u_{[1,a]} & = (a+1 \, a \, \cdots \, 3 \, 1 \, 2 \, a+2 \, a+3
      \, \cdots \, n)^{-1} \\ 
 & = a \, a+1 \, a-1 \, a-2 \, \cdots \, 2 \, 1 \, a+2 \, a+3 \,
 \cdots \, n
    \end{split}
  \end{equation}
and 
\begin{equation}
  \label{eq:v_[1,a]}
  \begin{split}
      v_{[1,a]} & = (2 \, a+1 \, a \, \cdots \, 4 \, 3 \, 1 \, a+2 \, a+3
      \, \cdots \, n)^{-1} \\ 
 & = a+1 \, 1\, a \, a-1 \, \cdots \, 3 \, 2 \, a+2 \, a+3 \,
 \cdots \, n
 \end{split}
\end{equation}
in one-line notation. For general subsets
\[
\mathcal{A} = [a_1, a_2] \cup [a_3, a_4] \cup \cdots \cup [a_{m-1},
a_m],
\]
with $a_1 = 1$ and $a_2 \geq 2$, 
the definitions $312$-type and $231$-type fixed points imply that 
\begin{equation}\label{eq:def uA}
u_{\mathcal{A}} = u_{[a_1, a_2]} w_{[a_3, a_4]} \cdots
w_{[a_{m-1},a_m]}
\end{equation}
and 
\begin{equation}
  \label{eq:def vA}
  v_{\mathcal{A}} = v_{[a_1, a_2]} w_{[a_3, a_4]} \cdots
w_{[a_{m-1},a_m]}.
\end{equation}

\begin{lemma}\label{lemma:initial segments and Aw}
Let $\mathcal{A},\mathcal{B}$ be subsets of $\{1,2,\ldots,n-1\}$ and
let 
\[
\mathcal{A} = [a_1, a_2] \cup [a_3, a_4] \cup \cdots \cup [a_{m-1},
a_m] \quad \textup{ and } \quad 
\mathcal{B} = [b_1, b_2] \cup [b_3, b_4] \cup \cdots \cup [b_{m-1},
b_m] 
\]
be the respective decompositions into maximal consecutive
substrings. Assume both $\mathcal{A}$ and $\mathcal{B}$ contain
$\{1,2\}$. Let $w_{\mathcal{A}}$ (respectively $v_{\mathcal{A}}$) be
the Peterson-type (respectively $231$-type) fixed point corresponding
to $\mathcal{A}$ and let $u_{\mathcal{B}}$ be the $312$-type fixed
point corresponding to $\mathcal{B}$. Then
\[
w_{\mathcal{A}} < u_{\mathcal{B}}  \textup{ (respectively
  $v_{\mathcal{A}} < u_{\mathcal{B}}$)  if and only if } \mathcal{A}
\subseteq \mathcal{B} \textup{ and } b_2 \geq a_2 +1.
\]
\end{lemma}

\begin{proof}
We begin by recalling two basic observations about Bruhat order
in $S_n$. Both follow straightforwardly from its definition in terms
of reduced word decompositions. Suppose $w, w' \in S_n$ and assume that $w$
and $w'$ do not share any simple transpositions in their reduced word
decompositions, i.e., $s_i < w$ implies $s_i \not < w'$ and vice
versa. Then firstly, $w \cdot w' < w''$ for $w'' \in S_n$ if and only
if both $w<w''$ and $w' < w''$. Secondly, $w< w' \cdot w''$ if and
only if $w < w''$. 

Recall that $w_{\mathcal{A}}$ can be written as 
\begin{equation}\label{eq:decomp wA}
w_{\mathcal{A}} = w_{[a_1,a_2]} \cdot w_{[a_3,a_4]} \cdots
w_{[a_{m-1},a_m]}.
\end{equation}
Moreover each factor appearing in the decomposition~\eqref{eq:decomp
  wA} (respectively~\eqref{eq:def vA} and~\eqref{eq:def uA}) for
$w_{\mathcal{A}}$ (respectively $v_{\mathcal{A}}$ and
$u_{\mathcal{B}}$) has the property that it does not share any simple
transpositions with any other factor appearing in the decomposition.

Now suppose $v_{\mathcal{A}}$ (respectively $w_{\mathcal{A}}$) is
Bruhat-less than $u_{\mathcal{B}}$. Then we know from
Lemma~\ref{lemma:Aw and transpositions} that $\mathcal{A} \subseteq
\mathcal{B}$ so it suffices to prove $b_2 \geq a_2+1$. From
Lemma~\ref{lemma:one line notation for nonPeterson} and the definition
of the Peterson type fixed points we know that the one-line
notation for $v_{\mathcal{A}}$ (respectively $w_{\mathcal{A}}$) has
first $a_2+1$ entries
\[
a_2+1 \, 1 \, a_2 \, a_2-1 \, \cdots \, 3 \, 2
\]
(respectively $a_2+1 \, a_2 \, a_2-1 \, \cdots \, 3 \, 2 \, 1$) while
the one-line notation of $u_{\mathcal{B}}$ has first $b_2+1$ entries
given by 
\[
b_2 \, b_2+1 \, b_2-1 \, \cdots \, 3 \, 2 \, 1.
\]
In particular $1 \in D_R(v_{\mathcal{A}})$ and also $1 \in
D_R(w_{\mathcal{A}})$. By the tableau criterion, this implies that the
first entry of the one-line notation of $v_{\mathcal{A}}$ and
$w_{\mathcal{A}}$ must be less than or equal to the first entry of
that of $u_{\mathcal{B}}$. Hence $a_2+1 \leq b_2$ as desired. 

Conversely suppose $\mathcal{A} \subseteq \mathcal{B}$ and $b_2 \geq
a_2+1$. Then an examination of the one-line notation of $v_{[a_1,
  a_2]}$ (respectively $w_{[a_1, a_2]}$) compared to that of $u_{[b_1,
  b_2]}$ and another application of the tableau criterion implies that
$v_{[a_1, a_2]} < u_{[b_1, b_2]}$ and $w_{[a_1, a_2]} < u_{[b_1,
  b_2]}$. In particular $v_{[a_1, a_2]}$ and $w_{[a_1, a_2]}$ are also
Bruhat-less than $u_{\mathcal{B}}$. Moreover since $\mathcal{A}
\subseteq \mathcal{B}$ it follows that $[a_3, a_4] \cup \cdots \cup
[a_{m-1},a_m] \subseteq \mathcal{B}$ so Lemma~\ref{lemma:wA maximal}
implies $w' := w_{[a_3, a_4]} \cdots w_{[a_{m-1},a_m]} <
w_{\mathcal{B}} = u_{\mathcal{B}} \cdot s_1$, where the last equality
follows from Lemma~\ref{lemma:one line notation for
  nonPeterson}. Since $s_1$ does not appear in any factor of $w'$ the
general fact above implies $w' < u_{\mathcal{B}}$. Finally since
neither $w_{[a_1, a_2]}$ nor $v_{[a_1, a_2]}$ share any simple
transpositions with $w'$ the other general fact above yields
$v_{\mathcal{A}} < u_{\mathcal{B}}$, $w_{\mathcal{A}} <
u_{\mathcal{B}}$ as desired. 

\end{proof}

\subsection{Proof of Proposition~\ref{proposition:upper vanishing}}

We may now prove the upper-triangular vanishing property of $334$-type
Hessenberg Schubert classes.

\begin{proof}[Proof of Proposition~\ref{proposition:upper vanishing}]

Let $w,w' \in \Hess(h)^{S^1}$ and let 
\[
\mathcal{A}(w) = [a_1, a_2] \cup [a_3, a_4] \cup \cdots \cup
[a_{m-1},a_m]
\quad \textup{ and } \quad 
\mathcal{A}(w') = [a'_1, a'_2] \cup [a'_3, a'_4] \cup \cdots \cup
[a'_{r-1}, a'_r]
\]
be the respective associated subsets decomposed into maximal
consecutive substrings.  By Lemmas~\ref{lemma:rolldowns bruhat less}
and~\ref{lemma:fixed points versus fillings} it suffices to prove
that if $\roll(w) \leq w'$, then $w \leq w'$. So suppose $\roll(w) \leq w'$. By
Lemma~\ref{lemma:Aw and transpositions} this implies $\mathcal{A}(w)
\subseteq \mathcal{A}(w')$. By Lemma~\ref{lemma:when set containment
  enough} we can conclude $w \leq w'$ if one of the following
hold:
\begin{itemize}
\item $w$ and $w'$ are of the same type, or 
\item $w$ is of Peterson type that does not contain $321$, and $w'$ is
  either $312$-type or $231$-type, or 
\item $w$ is either $312$-type or $231$-type, and $w'$ is of
  Peterson type. 
\end{itemize}

Now suppose one of the following holds:
\begin{itemize}
\item $w'$ is of Peterson type that does not contain $321$ and $w$ is not, 
\item $w'$ is $231$-type and $w$ is of Peterson type that contains
  $321$, or 
\item $w'$ is $231$-type and $w$ is $312$-type. 
\end{itemize}
In these cases, Lemma~\ref{lemma:not bruhat less} implies that $\roll(w)
\not < w'$ so there is nothing to prove.

It remains to discuss the cases when: 
\begin{itemize}
\item $w$ is of Peterson type that contains $321$ and $w'$ is
  $312$-type, or
\item $w$ is $231$-type and $w'$ is $312$-type. 
\end{itemize}
By Lemma~\ref{lemma:initial segments and Aw} it suffices to show that
$a'_2 \geq a_2+1$. Suppose $w$ is of Peterson type that contains
$321$. In particular $a_2 \geq 2$. From~\eqref{eq:rolldown one line
  notation for Peterson 321} we know that $1 \in D_R(\roll(w))$ and
the first entry in the one-line notation of $\roll(w)$ is $a_2+1$.  On
the other hand~\eqref{eq:u_[1,a]} implies the one-line notation of
$w'$ begins with $a'_2$.  So if $\roll(w) < w'$ then the tableau
criterion implies $a'_2 \geq a_2+1$ as desired.  Now suppose $w$ is of
$231$-type. Again $a_2 \geq 2$ and from~\eqref{eq:rolldown one line
  notation for 231} we know $\roll(w)$ has \(1 \in D_R(\roll(w))\) and
$a_2+1$ as its first entry.  By the same argument, $a'_2 \geq a_2+1$
as desired.  The result follows.

\end{proof}

\section{Combinatorial formulae for restrictions to fixed points of
  $334$-type Hessenberg Schubert classes}\label{sec:combinatorics}

Our goal in this section is to give a combinatorial formula for
$p_{\roll(w)}(w)$ from which it follows as a corollary that it is
nonzero. This proves Proposition~\ref{proposition:prollw at w
  nonzero} and hence Theorem~\ref{theorem:perfect}. Although not strictly necessary for the proof of
Proposition~\ref{proposition:prollw at w nonzero} we choose to prove
the explicit formula (Proposition~\ref{proposition:prollw at w} below)
since such a formula is a first step towards a
derivation of a Monk formula for $334$-type Hessenberg varieties and
because it conveys a flavor of
the combinatorics embedded
in the GKM theory of Hessenberg varieties which are larger than the
Peterson varieties in \cite{HarTym09}. 
Many of our
computations are analogues of those in
\cite[Section 5]{HarTym09}. Our main tool is \textbf{Billey's
  formula}. 
We briefly recall some definitions and results (see
also discussion in \cite[Section 4]{HarTym09}).

\begin{definition} \textbf{(\cite[Definition 4.7]{HarTym09})} 
Given a
  permutation $w \in S_n$, an index $j \in \{1,2,\ldots,\ell(w)\}$,
  and a choice of reduced word decomposition ${\bf b}=(b_1, b_2,
  \ldots, b_{\ell(w)})$ (corresponding to the word $w = s_{b_1}
  s_{b_2} \cdots s_{b_{\ell(w)}}$) for $w$, define
\begin{equation}\label{eq:def-beta}
r(j, \mathbf{b}) := s_{b_1} s_{b_2} \cdots s_{b_j-1} (t_{b_j} - t_{b_j+1}).
\end{equation}
\end{definition}
From the definition it follows that $r(j, \mathbf{b})$ is an element of $H^*_T(\pt) \cong
\Sym(\t^*) \cong \C[t_1, t_2, \ldots, t_n]$ of the form $t_\ell - t_k$
for some $\ell,k$.  These elements $r(j, \mathbf{b})$ are the building
blocks of Billey's formula \cite[Theorem 4]{Bil99} which computes the
restrictions $\sigma_v(w)$ of equivariant Schubert classes $\sigma_v$
at arbitrary permutations $w$ in $S_n$.

\begin{theorem}\label{theorem:billey} (\textbf{``Billey's formula'', \cite[Theorem 4]{Bil99}})
Let $w \in S_n$.  Fix a reduced word decomposition $w = s_{b_1}
s_{b_2} \cdots s_{b_{\ell(w)}}$ and let $\mathbf{b}= (b_1, b_2,
\ldots, b_{\ell(w)})$ be the sequence of its indices. 
Let $v \in S_n$. Then the restriction $\sigma_v(w)$ of the Schubert class $\sigma_v$ at
the $T$-fixed point $w$ is given by 
\begin{equation}\label{eq:billey-formula}
\sigma_v(w) = \sum r(j_1, \mathbf{b})r(j_2, \mathbf{b}) \cdots
r(j_{\ell(v)}, \mathbf{b})
\end{equation}
where the sum is taken over subwords $s_{b_{j_1}} s_{b_{j_2}} \cdots
s_{b_{j_{\ell(v)}}}$ of $\mathbf{b}$ that are reduced words for $v$.
\end{theorem}

We record the following fact, used in the proof below, which follows
straightforwardly from the
Billey formula.

\begin{fact}\label{fact:extra}
  Suppose $v, w \in S_n$ with $v \leq w$ in Bruhat order. Suppose
  there exists a decomposition $w = w' \cdot w''$ for $w', w'' \in
  S_n$ where $v \leq w'$ and, for all simple transpositions $s_i$ such
  that $s_i < v$, we have $s_i \not \leq w''$. Then $\sigma_v(w) =
  \sigma_v(w')$. 
\end{fact}

Following terminology in \cite{HarTym09}, we refer to an individual
summand of the expression in the right hand side of~\eqref{eq:billey-formula}, corresponding to
a single reduced subword $v=s_{b_{j_1}} s_{b_{j_2}} \cdots
s_{b_{j_{\ell(v)}}}$ of $w$, as a {\bf summand in Billey's formula}. 
In order to derive formulas for $p_v(w)$ where $p_v$ is a Hessenberg
Schubert class, we use the linear projection $\pi_{S^1}: \t^* \to \Lie(S^1)^*$
dual to the inclusion of our circle subgroup $S^1$ into $T$ given
by~\eqref{eq:def-circle}. More specifically, since the
diagram~\eqref{eq:comm diagram for Hess} commutes, we have  
\begin{equation}\label{eq:Billey's formula for pvw}
p_v(w) = \sum \pi_{S^1}(r(j_1, \mathbf{b})) \pi_{S^1}(r(j_2, \mathbf{b})) \cdots
\pi_{S^1}(r(j_{\ell(v)}, \mathbf{b})).
\end{equation}
We refer to the right hand side of the above equality as {\bf Billey's
  formula for} $p_v(w)$.  Recall $\pi_{S^1}(t_\ell - t_{k+1}) = (k+1-\ell)t$
for a positive root $t_\ell - t_{k+1}$ \cite[Section 5]{HarTym09}.

We also use the following.

\begin{definition}\label{defn:head} \textbf{(\cite[Definition 5.4]{HarTym09})}
Fix ${\mathcal{A}} \subseteq \{1,2,\ldots,n-1\}$.  Define
$\mathcal{H}_{\mathcal{A}}: \mathcal{A} 
\to \mathcal{A}$ by  
\[\head{{\mathcal{A}}}{j} = \textup{the maximal element in the maximal
  consecutive substring of ${\mathcal{A}}$ containing $j$}.\]
\end{definition}

\begin{definition}\label{defn:tail} \textbf{(\cite[Definition 5.5]{HarTym09})}
Fix ${\mathcal{A}} \subseteq \{1,2,\ldots,n-1\}$. Define $\mathcal{T}_{\mathcal{A}}: \mathcal{A}
\rightarrow \mathcal{A}$ by
\[\toe{{\mathcal{A}}}{j} = 
\textup{the minimal element in the maximal consecutive substring of
  ${\mathcal{A}}$ containing } j.\]
\end{definition}

We proceed to some preliminary computations.  Let $\mathbf{b} =
(b_1, \ldots, b_{\ell(w)})$ be a reduced word decomposition $w =
s_{b_1}s_{b_2} \cdots s_{b_{\ell(w)}}$ of $w$ and let $i$ be an index
appearing in $\mathbf{b}$, i.e. $b_\ell = i$ for some $1 \leq \ell
\leq \ell(w)$. Our first computation, Lemma~\ref{lemma:summands p si},
gives an expression for $\pi_{S^1}(r(\ell, \mathbf{b}))$ which shows
in particular that the value of $\pi_{S^1}(r(\ell, \mathbf{b}))$
depends only on the value of the index $b_\ell = i$ and not on its
location $\ell$ in the word $\mathbf{b}$. Note that if $v=s_i$ then
the summands in Billey's formula for $p_v(w) = p_{s_i}(w)$ are
precisely equal to $r(\ell, b)$ for each $\ell$ such that $b_\ell =
i$. Thus an equivalent formulation of the claim is that the summands
in Billey's formula for $p_{s_i}(w)$ are all equal. This is analogous
to a result in the Peterson case \cite[Lemma 5.2]{HarTym09} except
that in our situation, the form of the formulas depend on the index
$i$ as well as on the type of the fixed point $w$ in question.

\begin{lemma}\label{lemma:summands p si}
Let $w \in \Hess(h)^{S^1}$ and let 
$\mathbf{b} =
(b_1, \ldots, b_{\ell(w)})$ be the reduced word decomposition of $w$
chosen in Section~\ref{sec:upper triangularity}. 
Let $\mathcal{A}(w) = [a_1, a_2] \cup [a_3, a_4] \cup \cdots \cup
[a_{m-1}, a_m]$ be the associated subset of $w$ decomposed into
maximal consecutive substrings. Let $i \in
\{1,2,\ldots,n-1\}$.
\begin{enumerate} 
\item If $i \not \in \mathcal{A}(w)$, then each summand in Billey's
  formula for $p_{s_i}(w)$ is $0$. In particular, $p_{s_i}(w) = 0$. 
\item Suppose $i \in \mathcal{A}(w)$ and suppose one of the following conditions hold: 
\begin{itemize} 
\item $w$ is of Peterson type, or 
\item $w$ is $312$-type, or 
\item $w$ is $231$-type and $i \not \in [a_1, a_2]$. 
\end{itemize} 
Then each
  summand in Billey's formula for $p_{s_i}(w)$ is equal to 
\[
(i-\mathcal{T}_{\mathcal{A}(w)}(i)+1)t.
\]
\item Suppose $w$ is $231$-type and $i=1$. Then each summand in
  Billey's formula for $p_{s_i}(w)$ is equal to 
\[
a_2 t = \mathcal{H}_{\mathcal{A}(w)}(1) t. 
\]
\item Suppose $w$ if $231$-type and $i \in [2,a_2]$. Then each
  summand in Billey's formula for $p_{s_i}(w)$ is equal to 
\[
(i-\mathcal{T}_{\mathcal{A}(w)}(i))t = (i-1)t. 
\]
\end{enumerate} 
\end{lemma}

\begin{proof}
  If $i$ does not occur in $\mathcal{A}(w)$ then each summand is $0$
  by Billey's formula for $\sigma_{s_i}(w)$, since $s_i \not < w$ and
  thus never appears in the reduced word decomposition of $w$.  For
  the next claim, the fact that each summand is equal to
  $(i-\mathcal{T}_{\mathcal{A}(w)}(i)+1)t$ for the listed cases
  follows from examination of the chosen reduced word decompositions
  of $w$ and an argument identical to that in
  \cite{HarTym09}. Thus it remains to check the cases in which the
  summand differs from the case of Peterson varieties.  First suppose
  $w$ is $231$-type and that $i=1$. From the choice of explicit
  reduced word decomposition for such $w$ given in~\eqref{eq:231 fixed
    point reduced word} and Billey's formula, it follows that each
  summand in Billey's formula for $\sigma_{s_1}(w)$ is equal to
\begin{equation}\label{eq:summand for i=1}
\begin{split}
   r_2 (r_3 r_2) \cdots (r_{a_2-1} r_{a_2-2} \cdots r_3 r_2) (r_{a_2}
   r_{a_2-1} \cdots r_2 \hsm (t_1 - t_2)) &=    r_2 (r_3 r_2) \cdots
   (r_{a_2-1} r_{a_2-2} \cdots r_3 r_2)(t_1 - t_{a_2+1}) \\
 & = t_1 - t_{a_2+1}
\end{split}
\end{equation}
since the reflection $r_j$ switches $t_j$ and $t_{j+1}$. Hence we have
$p_{s_1}(w) = \pi_{S^1}(t_1-t_{a_2+1}) = (a_2+1-1)t = a_2 t =
\mathcal{H}_{\mathcal{A}(w)}(1) t$. 
Now suppose $w$ is $231$-type and $i \in
[2,a_2]$.  The factor in the reduced word decomposition~\eqref{eq:231
  fixed point reduced word} corresponding to $[1=a_1, a_2]$ is
equal to $w_{[2,a_2]} \cdot s_1$. By Fact~\ref{fact:extra},
for $i >1$ the presence of the extra $s_1$ does not
affect the Billey computation, so each summand is equal to that 
for the Peterson type fixed point $w_{[2,a_2]}$ and 
hence is equal to 
\[
(i-\mathcal{T}_{\mathcal{A}(w)}(i))) = (i-1)t,
\]
as desired. 
\end{proof}

Our next lemma concerns the summands in Billey's formula for
$p_{\roll(w)}(w)$ for $w \in \Hess(h)^{S^1}$.

\begin{lemma}\label{lemma:billey summands} 
Let $w \in \Hess(h)^{S^1}$. 
\begin{itemize}
\item Suppose $w$ is $312$-type or $w$ is of
Peterson type that does not contain $321$. 
Then each summand for Billey's formula for $p_{\roll(w)}(w)$
 is equal to 
\[
\left( \prod_{i \in
    \mathcal{A}(w)} (i - \mathcal{T}_{\mathcal{A}(w)}(i) + 1) \right)
\cdot t^{\lvert \mathcal{A}(w) \rvert}.
\]
\item Suppose $w$ is $231$-type. Then each summand for
  Billey's formula for $p_{\roll(w)}(w)$ is equal to 
\[
\head{\mathcal{A}(w)}{1} \cdot \left( \prod_{i =
    2}^{\head{\mathcal{A}(w)}{1}} (i-1) \right) \cdot \left( \prod_{i
    \in \mathcal{A}(w) \setminus [\toe{\mathcal{A}(w)}{1},
    \head{\mathcal{A}(w)}{1}]} (i - \toe{\mathcal{A}(w)}(i) +1)
\right) t^{\lvert \mathcal{A}(w) \rvert}.
\]
\item Suppose $w$ is of Peterson type that contains
  $321$. Then each summand for Billey's formula for $p_{\roll(w)}(w)$
  is equal to 
\[
\left( \prod_{i \in
    \mathcal{A}(w)} (i - \mathcal{T}_{\mathcal{A}(w)}(i) + 1) \right)
\cdot t^{\lvert \mathcal{A}(w) \rvert + 1}.
\]
\end{itemize}
\end{lemma}

\begin{proof}
  Before considering the separate cases we make a general observation.
  By Lemma~\ref{lemma:summands p si} and the discussion before
  Lemma~\ref{lemma:summands p si} we know that the summands in
  Billey's formula for $p_{s_i}(w)$ for $i \in \mathcal{A}(w)$ are
  exactly the terms $\pi_{S^1}(r(\ell,\mathbf{b}))$ for $\ell$ such
  that $b_\ell = i$. Suppose in addition that $w \in \Hess(h)^{S^1}$
  is such that 
$\roll(w)$ contains at most one
  simple transposition $s_i$ for each $i \in \{1,2,\ldots,n-1\}$,
  i.e., $\roll(w) = s_{i_1} s_{i_2} \cdots s_{i_{\ell(\roll(w))}}$ 
is a reduced word for $\roll(w)$ 
where all $i_k$ are distinct for $1 \leq k
\leq \ell(\roll(w))$. This implies that any subword of a reduced word
decomposition $\mathbf{b}$ of $w$ which is a reduced word of
$\roll(w)$ also must contain precisely one $s_{i_k}$ for each $1 \leq k
\leq \ell(\roll(w))$. From Billey's formula~\eqref{eq:Billey's formula
  for pvw} for
$p_{\roll(w)}(w)$ we
know that a summand is of the form 
\begin{equation}\label{eq:summand}
\pi_{S^1}(r(j_1, \mathbf{b})) \pi_{S^1}(r(j_2, \mathbf{b})) \cdots
\pi_{S^1}(r(j_{\ell(v)}, \mathbf{b}))
\end{equation}
where $s_{b_{j_1}} s_{b_{j_2}} \cdots s_{b_{j_{\ell(\roll(w))}}}$ is a
reduced word of $\roll(w)$. 
Since $\{b_{j_1},
\ldots, b_{j_{\ell(\roll(w))}}\} = \{i_1, i_2, \ldots,
i_{\ell(\roll(w))}\}$ for each such summand 
the quantity~\eqref{eq:summand} is equal to 
\begin{equation}\label{eq:summand in terms of p si}
\prod_{k=1}^{\ell(\roll(w))} p_{s_{i_k}}(w). 
\end{equation}

We now take cases. 
Suppose $w$ is not a Peterson-type that contains $321$. Then from the
explicit descriptions of $\roll(w)$ given in Section~\ref{sec:upper
  triangularity} it follows that $\roll(w)$ 
contains in its reduced word a single $s_i$ for each $i \in
\mathcal{A}(w)$. Thus we are in the situation described in the above
paragraph and the claims follow from the computations given in
Lemma~\ref{lemma:summands p si}.

Suppose $w$ is Peterson type and contains $321$.
Let 
$\mathbf{b}$ be the standard reduced word decomposition 
(cf.~\eqref{eq:standard reduced word}
and~\eqref{eq:wA-reduced-word}) of $w$. We claim
that the only reduced word decompositions
of $\roll(w)$
that occur as a subword of $\mathbf{b}$ are those which contain two
$s_1$'s, only one $s_2$, and precisely one $s_{j_\ell}$ for all other
$j_\ell$. 
Indeed let $\mathcal{A}(w) = [a_1, a_2] \cup [a_3, a_4]
\cup \cdots \cup [a_{m-1}, a_m]$ be the decomposition of
$\mathcal{A}(w)$ into maximal consecutive substrings. Recall $a_2 \geq
2$ and $a_1 = 1$ in this case. The rolldown
$\roll(w)$ is 
\[
(s_{a_m} s_{a_m-1} \cdots s_{a_{m-1}}) \cdots (s_{a_4} s_{a_4-1}
\cdots s_{a_3}) \cdot (s_{a_2} s_{a_2-1} \cdots s_1 s_2 s_1)
\]
and $w$ is 
\[
w = w_{[a_1, a_2]} w_{[a_3, a_4]} \cdots w_{[a_{m-1},a_m]}.
\]
Let $\ell>1$. There is only one reduced word decomposition of the
factor $s_{a_{\ell+1}} s_{a_{\ell+1}-1} \cdots s_{a_\ell}$ in
$\roll(w)$, so it remains to analyze the
subwords of $w_{[a_1,a_2]}$ which are reduced words of $s_{a_2}
s_{a_2-1} \cdots s_1 s_2 s_1$. Let $\mathbf{b}$ denote the standard
reduced word of $w_{[a_1, a_2]}$. Note that another valid reduced word
of $s_{a_2} s_{a_2-1} \cdots s_1 s_2 s_1$ is $s_{a_2} s_{a_2-1} \cdots
s_2 s_1 s_2$. Since $s_2$ does not commute with $s_1$, the
rightmost $s_2$ in the word $s_{a_2} s_{a_2-1} \cdots
s_2 s_1 s_2$ must appear to the right of the
$s_{a_2}$; in particular, there are two $s_2$'s to the right of the
$s_{a_2}$ in this word. Since there is only one $s_2$ appearing to the
right of the $s_{a_2}$ in $\mathbf{b}$
we conclude that the reduced word of $s_{a_2} s_{a_2-1}
\cdots s_1 s_2 s_1$ containing two copies of $s_2$ never appears as a
subword of $\mathbf{b}$. Hence the only subwords of $\mathbf{b}$
contributing to summands in Billey's formula for $p_{\roll(w)}(w)$
contain two $s_1$'s and one $s_2$, as claimed. 
Now since 
$j_1 = 1$ and $j_2=2$ and $p_{s_1}(w) = (1 -
\toe{\mathcal{A}(w)}{1}+1)t = (1-1+1)t = t$ by
Lemma~\ref{lemma:summands p si}, the claim follows. 
\end{proof}

We have just seen that all summands in Billey's formula for
$p_{\roll(w)}(w)$ are equal for all fixed points $w \in
\Hess(h)^{S^1}$. In order to finish the computation we must now
compute the number of summands which occur. 

\begin{lemma}\label{lemma:number of summands}
Let $w \in \Hess(h)^{S^1}$. 
\begin{itemize}
\item Suppose $w$ is of Peterson type that contains $321$. 
Then the number of summands in Billey's formula for $p_{\roll(w)}(w)$ is 
 $\mathcal{H}_{\mathcal{A}(w)}(1)-1$. 
\item Suppose $w$ is of Peterson type that does not contain $321$. Then the
 number
  of summands in Billey's formula for $p_{\roll(w)}(w)$ is $1$. 
\item Suppose $w$ is $312$-type. Then the number of summands in
  Billey's formula for $p_{\roll(w)}(w)$ is 
  $\mathcal{H}_{\mathcal{A}(w)}(1)-1$. 
\item Suppose $w$ is $231$-type. Then the number of summands in
  Billey's formula for $p_{\roll(w)}(w)$ is $1$. 
\end{itemize}
\end{lemma}

\begin{proof}
We consider each case in turn. Suppose $w$ is of Peterson type
that contains $321$. 
Let $\mathcal{A}(w) = [a_1, a_2] \cup [a_3, a_4]
\cup \cdots \cup [a_{m-1}, a_m]$ be the decomposition of
$\mathcal{A}(w)$ into maximal consecutive substrings. Recall $a_2 \geq
2$ and $a_1 = 1$ in this case. The rolldown
$\roll(w)$ is 
\[
(s_{a_m} s_{a_m-1} \cdots s_{a_{m-1}}) \cdots (s_{a_4} s_{a_4-1}
\cdots s_{a_3}) \cdot (s_{a_2} s_{a_2-1} \cdots s_1 s_2 s_1)
\]
and $w$ is 
\[
w = w_{[a_1, a_2]} w_{[a_3, a_4]} \cdots w_{[a_{m-1},a_m]}.
\]
Let $\ell>1$. 
As observed in the proof of Lemma~\ref{lemma:billey summands} 
there is only one reduced word decomposition of
the factor $s_{a_{\ell+1}} s_{a_{\ell+1}-1} \cdots s_{a_\ell}$ in
$\roll(w)$. 
Moreover by examination it is evident that it appears only
once in the standard reduced word decomposition of the corresponding
$w_{[a_\ell, a_{\ell+1}]}$ factor in $w$. 
Hence in order to count the number of
ways $\roll(w)$ appears in $w$ it suffices to count the number of
subwords of the standard reduced word decomposition $\mathbf{b}$ of 
$w_{[a_1, a_2]}$ which are reduced subwords of 
$s_{a_2} s_{a_2-1} \cdots s_1 s_2 s_1$. 
We already saw in the proof of
Lemma~\ref{lemma:billey summands} that the reduced word $s_{a_2} s_{a_2-1} \cdots s_2
s_1s_2$ never appears in $\mathbf{b}$. 
On the other hand since $s_1$ commutes with any $s_k$ with $k\geq
3$, another reduced word decomposition of $s_{a_2} s_{a_2-1}
\cdots s_1 s_2 s_1$ is $s_1 s_{a_2} s_{a_2-1} \cdots s_2 s_1$. 
From examination of $\mathbf{b}$ 
it can be seen that the word $s_1 s_{a_2} s_{a_2-1} \cdots s_2 s_1$
appears as a subword in the standard reduced word of $w_{[a_1,a_2]}$
precisely $a_2-1 = \head{\mathcal{A}(w)}{1}-1$ times and that these are the only subwords of
$\mathbf{b}$ which equal $s_{a_2} s_{a_2-1}
\cdots s_1 s_2 s_1$. The claim follows.

Suppose $w$ is of Peterson type that does not contain $321$. Then the
rolldown $\roll(w)$ is the Peterson case rolldown so the claim follows
from explicit examination of the standard reduced word of $w$
(alternatively from \cite[Fact 4.5]{HarTym09}).

Suppose $w$ is $312$-type. Then the rolldown $\roll(w)$ is of the
form 
\[
\roll(w) = (s_{a_m} s_{a_m-1} \cdots s_{a_{m-1}}) \cdots (s_{a_4} s_{a_4-1}
\cdots s_{a_3}) \cdot (s_{a_2} s_{a_2-1} \cdots s_1 s_2)
\]
from Lemma~\ref{lemma:rolldowns for nonPetersons}. By an argument similar to the case of
Peterson type that contains $321$ it suffices to analyze only the factors
in both $\roll(w)$ and $w$ corresponding to the initial maximal
consecutive substring $[a_1, a_2]$. As above we have 
\[
s_{a_2} s_{a_2-1} \cdots s_1 s_2 = 
s_1 s_{a_2} s_{a_2-1} \cdots  s_2
\]
and again 
it follows from examination of the standard reduced word of $u_{[a_1,
  a_2]}$ that $s_1 s_{a_2} s_{a_2-1} \cdots  s_2$ appears precisely
$a_2-1 = \head{\mathcal{A}(w)}{1}-1$ times. 

Finally suppose $w$ is $231$-type. Then the rolldown $\roll(w)$
coincides with the Peterson case rolldown of $w_{\mathcal{A}(w)}$ and
the claim follows from examination of the reduced word
decomposition~\eqref{eq:231 fixed point reduced word}. 
\end{proof}

The following is immediate from Lemmas~\ref{lemma:billey summands}
and~\ref{lemma:number of summands}.

\begin{proposition}\label{proposition:prollw at w}
Let $w\in \Hess(h)^{S^1}$. 
 \begin{itemize}
  \item Suppose $w$ is of Peterson type that contains $321$. Then 
\[
p_{\roll(w)}(w) = (\mathcal{H}_{\mathcal{A}(w)}(1)-1) \left( \prod_{i \in
    \mathcal{A}(w)} (i - \mathcal{T}_{\mathcal{A}(w)}(i) + 1) \right)
\cdot t^{\lvert \mathcal{A}(w) \rvert +1}.
\]
 \item Suppose $w$ is of Peterson type that does not contain $321$. Then 
\[
p_{\roll(w)}(w) = \left( \prod_{i \in \mathcal{A}(w)} (i -
  \mathcal{T}_{\mathcal{A}(w)}(i) + 1) \right) t^{\lvert
  \mathcal{A}(w) \rvert}.
\]
\item Suppose $w$ is of type $312$. Then 
\[
p_{\roll(w)}(w) = (\mathcal{H}_{\mathcal{A}(w)}(1)-1) \left( \prod_{i \in
    \mathcal{A}(w)} (i - \mathcal{T}_{\mathcal{A}(w)}(i) + 1) \right)
\cdot t^{\lvert \mathcal{A}(w) \rvert}.
\]
\item Suppose $w$ is of type $231$. Then 
\[
p_{\roll(w)}(w) = \head{\mathcal{A}(w)}{1} \cdot \left( \prod_{i =
    2}^{\head{\mathcal{A}(w)}{1}} (i-1) \right) \cdot \left( \prod_{i
    \in \mathcal{A}(w) \setminus [\toe{\mathcal{A}(w)}{1},
    \head{\mathcal{A}(w)}{1}]} (i - \toe{\mathcal{A}(w)}(i) +1)
\right) t^{\lvert \mathcal{A}(w) \rvert}.
\]
  \end{itemize}
\end{proposition}

The proofs of the main results are now immediate. 

\begin{proof}[Proof of Proposition~\ref{proposition:prollw at w
    nonzero}]
Let $w \in \Hess(h)^{S^1}$. From the explicit formulas given in
Proposition~\ref{proposition:prollw at w} it follows that
$p_{\roll(w)}(w) \neq 0$ for all possible types of fixed points $w$. 
\end{proof}

\begin{proof}[Proof of Theorem~\ref{theorem:perfect}]
Since both~\eqref{eq:prollw at w nonzero} and~\eqref{eq:upper vanishing} are satisfied for all $w, w'
\in \Hess(h)^{S^1}$ by Propositions~\ref{proposition:upper vanishing}
and~\ref{proposition:prollw at w nonzero} respectively, the result
follows. 
\end{proof}

\section{Open questions}\label{sec:questions}

This manuscript raises more questions than it answers. We close by
mentioning some of them.

\begin{question}
For $n \geq 4$, Theorem~\ref{theorem:perfect} shows that for the case when $N:\C^n \to \C^n$ is the
principal nilpotent operator and $h$ is the $334$-type Hessenberg
function, the dimension pair algorithm produces a set of Hessenberg
Schubert classes $\{p_{\roll(w)}\}_{w \in \Hess(N,h)^{S^1}}$ which are
poset-upper-triangular and hence 
form a $H^*_{S^1}(\pt)$-module basis for $H^*_{S^1}(\Hess(N,h))$. 
\begin{enumerate} 
\item What are other examples of $N$ and $h$ such that
  the conclusion of Lemma~\ref{lemma:injective} holds
  (cf. Remark~\ref{remark:injective})? 
\item What are other examples of $N$ and $h$ for which the dimension
  pair algorithm produces a successful outcome of Betti poset pinball  which is also 
  poset-upper-triangular? Are there necessary and sufficient
  conditions on $N$ and $h$ that guarantee poset-upper-triangularity? 
\item What are other examples of $N$ and $h$ for which the dimension
  pair algorithm produces a successful outcome of Betti poset pinball  which
  corresponds to a linearly independent set of classes and hence a
  module basis? 
Are there necessary and sufficient conditions on $N$ and $h$ that
guarantee this? 
\end{enumerate} 
\end{question}

\begin{question}
In \cite{HarTym09} the explicit module basis consisting of Peterson
Schubert classes is used to derive a \textbf{manifestly positive Monk
  formula} in the $S^1$-equivariant cohomology of Peterson varieties.
Preliminary investigation suggests that an analogous Monk formula for
the $334$-type Hessenberg varieties, using the module basis of
Hessenberg Schubert classes derived in this manuscript, would be
computationally much more complex. Thus, we may ask the following. 
\begin{enumerate} 
\item Does there exist 
a combinatorially elegant or computationally effective Monk
formula for the $334$-type Hessenberg varieties? 
\item Can such a Monk formula be further generalized to a larger
  family of regular nilpotent Hessenberg varieties? For instance, 
can our techniques be generalized to give new insights to the equivariant Schubert calculus of the full flag variety $\Flags(\C^n)$ (which is an example of a regular nilpotent Hessenberg variety)? 
\item In \cite{BayHar10a} the Monk formula for Peterson varieties is
  used to derive a \textbf{Giambelli formula}. Does there also exist
  a combinatorially elegant and/or computationally effective Giambelli
  formula for other cases of regular nilpotent Hessenberg varieties? 
\end{enumerate} 
\end{question}

\def\cprime{$'$}


\begin{thebibliography}{10}

\bibitem{BayHar10a}
D.~Bayegan and M.~Harada.
\newblock A {G}iambelli formula for the {$S^1$}-equivariant cohomology of type
  {$A$} {P}eterson varieties, arXiv:1012.4053.

\bibitem{Bil99}
S.~Billey.
\newblock Kostant polynomials and the cohomology ring of {G/B}.
\newblock {\em Duke Math. J.}, 96:205--224, 1999.

\bibitem{BjoBre05}
A.~Bj{\"o}rner and F.~Brenti.
\newblock {\em Combinatorics of {C}oxeter groups}, volume 231 of {\em Graduate
  Texts in Mathematics}.
\newblock Springer, New York, 2005.

\bibitem{BriCar04}
M.~Brion and J.~B. Carrell.
\newblock The equivariant cohomology ring of regular varieties.
\newblock {\em Michigan Math. J.}, 52(1):189--203, 2004.

\bibitem{CarrellKaveh:2008}
J.~B. Carrell and K.~Kaveh.
\newblock On the equivariant cohomology of subvarieties of a {$B$}-regular
  variety.
\newblock {\em Transform. Groups}, 13(3-4):495--505, 2008.

\bibitem{DeMProSha92}
F.~De~Mari, C.~Procesi, and M.~A. Shayman.
\newblock Hessenberg varieties.
\newblock {\em Trans. Amer. Math. Soc.}, 332(2):529--534, 1992.

\bibitem{DewHar10}
B.~Dewitt and M.~Harada.
\newblock Poset pinball, highest forms, and {$(n-2,2)$} {S}pringer varieties, arXiv:1012.5265.


\bibitem{Ful99}
J.~Fulman.
\newblock Descent identities, {H}essenberg varieties, and the {W}eil
  conjectures.
\newblock {\em J. Combin. Theory Ser. A}, 87(2):390--397, 1999.

\bibitem{Fun03}
F.~Y.~C. Fung.
\newblock On the topology of components of some {S}pringer fibers and their
  relation to {K}azhdan-{L}usztig theory.
\newblock {\em Adv. Math.}, 178(2):244--276, 2003.

\bibitem{HarTym10}
M.~Harada and J.~Tymoczko.
\newblock Poset pinball, {GKM}-compatible subspaces, and {H}essenberg
  varieties, arXiv:1007.2750.

\bibitem{HarTym09}
M.~Harada and J.~Tymoczko.
\newblock A positive {M}onk formula in the {$S^1$}-equivariant cohomology of
  type {A} {P}eterson varieties, arXiv:0908.3517.
\newblock To be published in Proc. London Math. Soc.

\bibitem{Kos96}
B.~Kostant.
\newblock Flag manifold quantum cohomology, the {T}oda lattice, and the
  representation with highest weight {$\rho$}.
\newblock {\em Selecta Math. (N.S.)}, 2(1):43--91, 1996.

\bibitem{Mbirika:2010}
A.~Mbirika.
\newblock {A {H}essenberg generalization of the {G}arsia-{P}rocesi basis for
  the cohomology ring of {S}pringer varieties}.
\newblock {\em Electronic Journal of Combinatorics}, 17, R153 (electronic), 2010.  

\bibitem{Rie03}
K.~Rietsch.
\newblock Totally positive {T}oeplitz matrices and quantum cohomology of
  partial flag varieties.
\newblock {\em J. Amer. Math. Soc.}, 16(2):363--392 (electronic), 2003.

\bibitem{Shi85}
N.~Shimomura.
\newblock The fixed point subvarieties of unipotent transformations on the flag
  varieties.
\newblock {\em J. Math. Soc. Japan}, 37(3):537--556, 1985.

\bibitem{Spa76}
N.~Spaltenstein.
\newblock The fixed point set of a unipotent transformation on the flag
  manifold.
\newblock {\em Nederl. Akad. Wetensch. Proc. Ser. A {\bf 79}=Indag. Math.},
  38(5):452--456, 1976.

\bibitem{Tym06}
J.~S. Tymoczko.
\newblock Linear conditions imposed on flag varieties.
\newblock {\em Amer. J. Math.}, 128(6):1587--1604, 2006.

\bibitem{Tym05}
J.~S. Tymoczko. 
\newblock An introduction to equivariant cohomology and homology, following {G}oresky, {K}ottwitz, and {M}ac{Pherson}. 
\newblock Snowbird lectures in algebraic geometry, {\em Contemp. Math.}, 388:169--188, 2005. 

\end{thebibliography}
\end{document}